\documentclass [12pt,a4paper,reqno]{amsart}
\input amssymb.sty
\newcommand\Label[1]{\label{#1}}

\usepackage[OT2,OT1]{fontenc}
\newcommand\cyr{%
\renewcommand\rmdefault{wncyr}%
\renewcommand\sfdefault{wncyss}%
\renewcommand\encodingdefault{OT2}%
\normalfont\selectfont}
\DeclareTextFontCommand{\textcyr}{\cyr}

\newtheorem{thm}{Theorem}[section] 
\newtheorem{constr}[thm]{Construction}

\newtheorem{rem}[thm]{Remark}

\newtheorem{prop}[thm]{Proposition}

\newtheorem{exmpl}[thm]{Example}

\newtheorem{cor}[thm]{Corollary}
\newtheorem{lem}[thm]{Lemma}
\newtheorem{defn}[thm]{Definition}

\newtheorem*{thm*}{Theorem}

\newtheorem{prop*}{Proposition}

\newtheorem*{examp*}{Example}
\newtheorem*{examples*}{Examples}
\newtheorem*{remark*}{Remark}
\newtheorem*{Note*}{Note}
\newtheorem*{defn*}{Definition}

\newtheorem*{note*}{Note}

\newcommand\Tref[1]{{Theorem~\ref{#1}}}
\newcommand\Dref[1]{{Definition~\ref{#1}}}
\newcommand\Rref[1]{{Remark~\ref{#1}}}
\newcommand\Cref[1]{{Corollary~\ref{#1}}}
\newcommand\Sref[1]{{Section~\ref{#1}}}
\newcommand\Pref[1]{{Proposition~\ref{#1}}}
\newcommand\Lref[1]{{Lemma~\ref{#1}}}
\newcommand\Eref[1]{{Example~\ref{#1}}}

\def\len{{\operatorname{len}}}

\newcommand\Cent[1]{{\operatorname{Cent}(#1)}}

\newcommand{\card}[1]{{\left|{#1}\right|}}

\renewcommand\th[1]{{${#1}^{\rm{th}}$}}
\def\id{{\operatorname {id}}}
\def\Ann{{\operatorname {Ann}}}
\def\inv{^ {-1}}
\def\An{{\mathbf A}^n}
\def\sub{\subseteq}

\newcommand\M[1][d]{{\operatorname{M}_{#1}}}

\newcommand\GL[1][d]{{\operatorname{GL}_{#1}}} 

\newcommand\dimcol[2]{{[{#1}\!:\!{#2}]}}

\newcommand\eq[1]{{(\ref{#1})}}
\newcommand\Eq[1]{{Equation~\eq{#1}}}

\def\normali{{\lhd}} 

\def\isom{{\;\cong\;}} 
\newcommand{\set}[1]{{\left\{#1\right\}}}

\def\lam{{\lambda}}
\newcommand\mul[1]{{#1^{\times}}} 
\def\({\left(}
\def\){\right)}

\def\co{{\,{:}\,}}
\def\ra{{\rightarrow}}

\newcommand\algint[2][]{\if!#1\relax O_{#2} \else{O_{\!#2}^{\phantom{I}}}\fi}

\newcommand\suchthat{{\,:\ \,}}

\def\N{\mathbb N}
\def\la{\lambda}
\def\Z{\mathbb Z}
\def\R{\mathbb R}

\def\cha{\operatorname{Char}}
\def\pol{\operatorname{poly}}

\def\N{{\mathbb {N}}}
\def\F{{\mathbb {F}}}
\def\C{{\mathbb {C}}}

\def\CS{{\mathcal C}}
\def\s{\sigma}
\def\a{\alpha}

\def\ie {{i.e.}}
\def\eg {{e.g.}}
\def\cf {{cf.}}

\newcommand\cl[2][]{{\if!#1!{#2^{\operatorname{cl}}}\else{#2}_{#1}^{\operatorname{cl}}\fi}}

\newcommand\Var[2][]{{\operatorname{Var}_{#1}(#2)}}

\newcommand\VarF[2][]{{\operatorname{Var_F}_{#1}(#2)}}

\newcommand\Rad{{{\operatorname{Rad}}}}

\newcommand{\smat}[4]{{\(\!\!\begin{array}{cc}
    {#1}\!&\!{#2}\\[-0.1cm]{#3}\!&\!{#4}\end{array}\!\!\)}} 

\newcommand\AR[1]{\begin{matrix}#1\end{matrix}}

\newcommand\defin[1]{{\bf {#1}}}

\def\Zcd{{Zariski-closed}}
\def\Zcr{{Zariski closure}}
\def\fcr{{linear closure}} 

\author{Alexei Belov-Kanel}
\address{Department of Mathematics, Bar-Ilan University, Ramat-Gan
52900, Israel} \email{belova@macs.biu.ac.il}
\author{Louis Rowen}
\address{Department of Mathematics, Bar-Ilan University, Ramat-Gan 52900,
Israel} \email{rowen@macs.biu.ac.il}
\author{Uzi Vishne}
\address{Department of Mathematics, Bar-Ilan University, Ramat-Gan 52900, Israel}
\email{vishne@macs.biu.ac.il}

\date{September 22, 2008}

\title[Zariski-closed algebras] 
{Structure of Zariski-closed algebras}

\thanks{This research was supported by the Israel Science
Foundation, grant \#1178/06.}


\begin{document}

\begin{abstract}
The objective of this paper is to describe the structure of \Zcd\
algebras, which provide a useful generalization to finite
dimensional algebras in the study of representable algebras over
finite fields. Our results include a version of Wedderburn's
principal theorem, as well as a more explicit description using
representations, in terms of ``gluing'' in Wedderburn components.
Finally, we construct ``generic'' \Zcd\ algebras, whose
description is considerably more complicated than the description
of generic algebra of finite dimensional algebras.

Special attention is given to infinite dimensional algebras over
finite fields.

\end{abstract}

\maketitle

\section{Introduction}\label{s:intro}
 This paper grew out of work on algebras satisfying a polynomial
identity (PI). We recall \cite[pp.~28ff.]{BR} that a PI-algebra
$R$ over an integral domain $C$ is {\bf representable} if it can
be embedded as a subalgebra of $\M[n](K)$ for a suitable field
$K\supset C$ (which can be much larger than $C$). One main
byproduct of Kemer's theorem \cite[Corollary 4.67]{BR} is that
every relatively free affine PI-algebra over an infinite field is
representable. From this perspective, the proof of Kemer's theorem
is based on a close study of representable algebras. The strategy
is to find the PI-algebra with the ``best'' structure, that is
PI-equivalent to a given representable algebra, in order to study
its identities very carefully. (Note that in characteristic 0 for
the non-affine case, Kemer proved that any relatively free algebra
can be embedded in the Grassmann envelope of a finite dimensional
superalgebra, so similar considerations also hold in this
case.)

Whereas over an infinite field, any representable algebra is
PI-equivalent to a finite dimensional $K$-algebra (thus leading to
a very careful study of identities of finite dimensional algebras
in the proof of Kemer's theorem), this is no longer the case over
finite fields (in positive characteristic). Thus, we need to
replace finite dimensional algebras by a more general class,
called {\bf \Zcd\ algebras}, which, surprisingly, satisfy much of
the structure theory of finite dimensional algebras. Since the
relatively free affine algebra of an affine PI-algebra is
representable, we are led finally to study the \Zcr\ of a
(representable) relatively free algebra.

Throughout the paper, $F \sub K$ will be fields, with $F$ finite
or infinite and $K$ usually being algebraically closed; $A$ is an
$F$-algebra contained in a finite dimensional $K$-algebra $B$. We
usually assume that $F$ has characteristic $p>0$, since the theory
becomes standard in characteristic 0.

After some introductory comments in Section \ref{sec2}, we introduce the
{\bf \Zcr} of a representable algebra $A$ in Section \ref{s:zcr}, showing
that it shares many of the important structure theorems of finite
dimensional algebras, such as Wedderburn's principal theorem and
the fact that every semiprime \Zcd\ algebra is semisimple; it
turns out that \Zcd\ algebras are semiperfect. Identities and
defining relations of $A$ also are studied in terms of its \Zcr\
in $B$, to be defined below.

In Section \ref{sec:3} we delve more deeply into the generation of
polynomial relations of a \Zcd\ algebra, showing that the center
is defined in terms of finitely many polynomial relations, which
can be written in the form  $\la _i = 0$, $\la _i- \la _i^s = 0$,
where $s$ is a $p$-power, or
 $\la _i - \la _j^s = 0$, $j \ne i$, where $s$ is a
$p$-power. These polynomial relations are said to be of {\it
Frobenius type}. This enables us explicitly to study
representations of \Zcd\ algebras in \Sref{sec:4}, focusing on
their Peirce decomposition, and its refinements. The explicit
representation of algebras is complicated even in characteristic
0, and one of our main techniques is ``gluing,'' or identifying
different components in a representation.

In \Sref{s:explicit}, we also obtain results concerning the
off-diagonal polynomial relations, which requires us to consider
 \defin{$q$-polynomials}, which we call polynomial relations of
{\it weak Frobenius type}. The main result is that the weak
Frobenius relations comprise a free module over the group algebra
of the Frobenius automorphism. We thank B.~Kunyavskii for bringing to
our attention the references \cite{KombMiyanMasayoshi},
\cite{Miyanishi}, and \cite{Tits}.

Finally, in  \Sref{sec:6} we describe   the relatively free
algebras of \Zcd\ algebras. These turn out to have an especially
nice description and play a key role in the proof of Specht's
conjecture for affine PI-algebras of arbitrary characteristic.

\section{Background}\label{sec2}
Let us bring in the main tools for our study.

\subsection{Results from the theory of finite dimensional
algebras}

We start with a classical theorem of Wedderburn about finite
dimensional algebras:

\begin{thm}[Wedderburn's Principal Theorem] \Label{Wed2}
Any finite dimensional algebra $A$ over a perfect field $F$ has a
Wedderburn decomposition  $A = S\oplus J$, where $J$ is the
Jacobson radical of $A$, which in this case is also the largest
nilpotent ideal, and $S \cong A/J$ is a semisimple subalgebra
of $A$.%
\end{thm}


When the base field is algebraically closed, Wedderburn's
Principal Theorem enables us to find a direct product of matrix
rings inside any finite dimensional algebra $A$. The following
notion helps us to better understand the structure of~$A$.

We call $\{e_1,\dots, e_n \}$  a {\bf 1-sum set} of orthogonal
idempotents if they are orthogonal and $\sum _{i=1}^n e_{i} = 1$.

\begin{rem}[``Peirce decomposition'']
If $A$ has  a {\bf 1-sum set} of orthogonal idempotents
$\{e_1,\dots, e_n\}$ (i.e., $e_ie_j = 0$ for all $i\ne j$), then
$$A\ = \bigoplus_{i,j=1}^t  e_iAe_j $$ as
additive groups.
\end{rem}

For example, the Peirce decomposition of $A = \M[n](R)$ with
respect to the matrix units $e_{11}, e_{22}, \dots, e_{nn}$ is
just
\begin{equation}\Label{PDM}
A= \bigoplus _{i,j = 1}^n Ae_{ij}.
\end{equation}

Note that any set $\{e_1,\dots , e_n\}$ of orthogonal idempotents
of $A$ can be expanded to a 1-sum set $\{e_0, e_1,\dots , e_n\}$
by taking $e =  \sum _{i=1}^n e_i$ and putting $e_{0} = 1 -e$.

Even for algebras without $1$, one can reproduce an analog of the
Peirce decomposition by formally defining a left and right
operator $e_0$  from $A$ to $A$, given by
$$e_0a = a -ea, \qquad ae_0 = a - ae.$$

\subsection{Affine varieties and algebraic groups}

We need some basic facts from affine algebraic geometry and the
theory of affine algebraic groups. We use \cite{Hum} as a
reference for algebraic groups. Suppose $K$ is an algebraically
closed field. Write $K[\Lambda]$ for the polynomial algebra $
K[\la_1, \dots, \la_n]$. For any subset $S\subset K[\Lambda]$, we
define the {\bf zero set} of $S$ to be
$$\mathcal Z (S) = \{ {\mathbf a}= (\a_1, \dots, \a_n ) \in K^{(n)}:
f(\a_1, \dots, \a_n) = 0, \,\forall f \in S\}.$$ %
$K^{(n)}$ has the
{\bf Zariski topology} whose closed sets are the zero sets of
subsets of $K[\Lambda]$.  This is the smallest topology under
which all polynomial maps $K^{(n)} \ra K$ are continuous, assuming
$K$ has the co-finite topology.

A closed set is {\bf irreducible} if it is not the union of two
proper closed subsets. An {\bf affine variety} is a \Zcd\ subset
of $K^{(n)}$. The {\bf dimension} of a variety is the length of a
maximal chain of irreducible subvarieties, with respect to
(proper) inclusion. A {\bf morphism} of varieties is a continuous
function with respect to the respective topologies. (In this text
we concern ourselves only with affine varieties, so ``variety''
means ``affine variety.'')

A {\bf locally closed} set is the intersection of a closed set and
an open set. A {\bf constructible set} is the finite union of
locally closed sets. We need the following theorem of Chevalley:

\begin{thm}[{\cite[Theorem 4.4]{Hum}}]\Label{Chev}
Any morphism of varieties sends constructible sets to
constructible sets. (In particular, the image of a variety is
constructible.)\end{thm}

An {\bf (affine) algebraic group} is an (affine) variety $G$
endowed with a group structure $(G,\cdot,e)$ such that the inverse
operation (given by $g \mapsto g\inv$) and multiplication map $ G
\times G \to G$ (given by $ (a,b) \mapsto a\cdot b)$ are morphisms
of varieties. A {\bf morphism} $ \varphi\co G\to H$ of algebraic
groups is a group homomorphism that is also a morphism of
varieties.

\begin{thm}[{\cite[Proposition 7.3]{Hum}}] In any algebraic group $G$, the
irreducible component $G_e$ of the identity  is a closed connected
subgroup of finite index, whose cosets  are precisely the
(connected) irreducible components of $G$. Thus, as a variety, $G$
is the direct product of an irreducible variety and a finite set.
\end{thm}

By \cite[Theorem 11.5]{Hum}, for any affine algebraic group $G$
with closed normal subgroup $N$, the group $G/N$ can be provided
with the structure of an algebraic group.

\subsection{Frobenius automorphisms and finite fields}\Label{ss:ff}

Much of our theory depends on the properties of endomorphisms of
finite fields. Towards this end, we recall the {\bf Frobenius
endomorphism} of a field $F$ of characteristic $p$ given by $a
\mapsto a^{p^t}$ for suitable fixed $t$. When $F$ is finite, then
every algebra endomorphism of $F$ is obviously an automorphism
(over its characteristic subfield), and it is well known by Galois
theory that every automorphism of $F$ is Frobenius.

When $F$ is an infinite field, there may of course be
non-Frobenius endomorphisms (but one can show using a Vandermonde
matrix argument, for any automorphism $\s$, that if $\s (a)$ and
$a$ are algebraically dependent of bounded degree for all $a \in
F$, then $\s$ is a Frobenius endomorphism).

Note that the Frobenius endomorphism of an algebraically closed
field $K$ also is an automorphism of $K$, although $K$ is
infinite.

\begin{thm}[Wedderburn's theorem about finite division rings]
\Label{Wed3} %
Any finite division ring is commutative. Consequently, any finite
dimensional simple algebra over a finite field $F$ must have the
form $\M[n](F_1)$ for a finite extension $F_1$ of $F$.
\end{thm}

Any finite field $F$ can be viewed as the zero set of the
polynomial $\la ^q - \la$ in its algebraic closure $K$, where $q =
\card{F}$. This observation enables us to view finite fields
explicitly as subvarieties (of dimension $0$) of the affine line.
Likewise, matrices over finite fields can be viewed naturally as
varieties.

\subsection{Examples of representable PI-algebras over finite and
infinite fields}

A polynomial identity (PI) of an algebra $A$ is a polynomial which
vanishes identically for any substitution in $A$. Recall that a
ring $R$ is called a {\bf central extension} of a subring $A$ if
$R = \Cent{R}A$. If $A$ is an algebra over an infinite field, then
any central extension of $A$ is PI-equivalent to $A$;
\cf~\cite[Proposition 1.1.32]{Row1}. Thus, in the examples to
follow, the finiteness of the field~$F$ is crucial for their
special properties concerning identities.

\begin{exmpl}\Label{basexa1} Suppose $F \sub K$ are fields.
\begin{enumerate}
\item \Label{BE1i}
Let $A = \smat{F}{K}{0}{F}$ (which is an $F$-algebra but not a
$K$-algebra). Then $\smat{K}{K}{0}{K}$ is a central extension of
$A$ since $A$ contains the matrix units $e_{11}, e_{12}$, and
$e_{22}$. When $F$ is infinite, $A$ is PI-equivalent to
$\smat{K}{K}{0}{K}$. However, when $\card{F} = q$ is finite, then
$\a^q = \a$ for all $\a \in F$, implying $a^q -a \in
\smat{0}{K}{0}{0}$
 for $a \in A$. Hence $(x^q-x)(y^q-y)\in \id(A)$.

\item \Label{BE1ii}
Let $A = \smat{F}{K}{0}{K}$,  where $\card{F} = q$. Then $a^q -a
\in \smat{0}{K}{0}{K}$, for all $a \in A$, implying $(x^q-x) [y,z]
\in \id (A)$.

\item \Label{BE1iii}
Let $A = \smat{K}{K}{0}{F}$,  where $\card{F} = q$. Then,
analogously to (\ref{BE1ii}), $[y,z] (x^q-x) \in \id (A)$.
\end{enumerate}
\end{exmpl}
There is another type of example, involving identification of
elements.

\begin{exmpl}\Label{basexa2}
Suppose $\s$ is an automorphism of $F_1$ over $F$, where $F \sub
F_1 \sub K$. Then $K$ can be viewed as an $F_1$-left module in the
usual way and as a right module ``twisted'' by $\s;$ namely
$a\cdot \a$ is defined as $a\s^{-1} (\a)$ for $a\in K$, $\a \in
F$. (We denote this new right module structure as $K_\s$.) Then $
\smat{F_1}{K_\s}{0}{F_1}$ is a PI-algebra, which is clearly
isomorphic to $\smat{F_1}{K}{0}{F_1}$ as a ring (but not as an
$F$-algebra). However, we get interesting new examples by making
certain identifications.

\begin{enumerate}
\item\Label{BE2i}
Suppose $\card{F_1} = q^t$, where $\card{F} = q$. Then we have the
Frobenius automorphism $\a \mapsto \a^{q^n}$ of $F_1$, and $\set{
\smat{\alpha^{p^n}}{a}{0}{\alpha}: \alpha \in F_1, \, a\in K}$
satisfies the identity $x[y,z] = [y,z]x^{p^n}$. Note that this
$F$-algebra is not an $F_1$-algebra in general.

\item\Label{BE2ii} Let $A = \set{ \smat{\s(\a)}{a}{0}{\a}: \a \in F_1, \ a
\in K }$. As a consequence of Theorem~\ref{linear} to be proved
below, if $\s$ is not Frobenius, then $\id (A) = \id(T_2)$, where
$T_2$ is the algebra of $2\times 2$ triangular matrices.
\end{enumerate}
\end{exmpl}

We call this identification process \textbf{gluing}, and it will be
described more precisely in \Sref{sec:4}. All of
Example~\ref{basexa1} and Example~\ref{basexa2} have a
central extension to $B= \smat{K}{K}{0}{K}$, and thus they satisfy the
same multilinear identities of $K$. But these varieties are quite
different. Thus, as opposed to algebras over infinite fields, in
general the multilinear identities are far from describing the
full PI picture.

For later use, we record the following result.
\begin{prop}\Label{break}
If $A = A_1 + A_2$, then a non-commutative polynomial $f$ is an
identity of $A$ iff $f$ and its consequences become zero under
substitutions in which every variable takes values either in $A_1$
or in $A_2$.
\end{prop}
\begin{proof}
This is trivial is characteristic zero, where every identity is
equivalent to a set of multilinear ones. In general, the proof is
by induction on the degree of $f$, considering the
multilinearization $f(\vec{x}+\vec{y}) - f(\vec{x}) - f(\vec{y})$.
\end{proof}

\section{The \Zcr\ of a representable algebra}\Label{s:zcr}

Both the motivation for PI-theory and one of its major facets is
the theory of representable algebras. In this section we develop
this theory, with emphasis always on the set of identities of a
given representable algebra $A$. Thus we often exchange $A$ by an
appropriate PI-equivalent algebra.

Let $F$ be a field and $A$ an arbitrary $F$-algebra. Recall from
the introduction that $A$ is representable if $A$ embeds (as an
$F$-algebra) in $\M[n](K)$ for a suitable extension field $K$
(possibly infinite dimensional) of $F$ and suitable $n$. In this
section we assume throughout that $A$ is representable. Then $A$
can be embedded further in $B = \M[n](\bar K)$, where $\bar K$ is
the algebraic closure of $K$, so we  assume throughout, without
loss of generality, that $K$ is algebraically closed. Thus, we
view $\M[n](K)$ as an $n^2$-dimensional variety and have the
theory of affine algebraic geometry at our disposal.

When the base field $F$ is infinite, $A$ is PI-equivalent to the
$K$-subalgebra $KA$ of $\M[n](K)$, which is finite dimensional, so
one passes at once to the finite dimensional case over an
algebraically closed field. In other words, one considers finite
dimensional algebras over a field, in which case one has the tools
from the theory of finite dimensional algebras, as described
above.

However, over finite fields (which clearly have positive
characteristic), it does not suffice to consider $K$-subalgebras
of $\M[n](K)$, as evidenced in Example~\ref{basexa1}, where we
have examples of algebras $A$ for which $KA = \smat{K}{K}{0}{K}$,
but $A$ satisfies extra identities. Thus we need a subtler way,
not passing all the way to the algebraic closure, of obtaining
``canonical'' algebras that are PI-equivalent to a given
representable algebra.

Our solution is to consider the \Zcr\ of $A$ in $\M[n](K)$, which
enjoys the analogs of all of the properties of finite dimensional
algebras listed above.

To show that an $F$-algebra  $A$ is representable, it clearly is
enough to embed $A$ into any finite dimensional unital $K$-algebra
$B$, since letting $n = \dimcol{B}{K}$ we can further embed $B$
into $\M[n](K)$. So we consider this situation that $A \subseteq
B$, where $B$ is an $n$-dimensional algebra over the algebraically
closed field $K$. At first, we assume that $B$ is a matrix
algebra, but later we modify our choice of~$B$ to better reflect
the structure of $A$.

\subsection{The \Zcr}
\begin{defn}
Suppose $B$ is a $K$-vector space, with $\dimcol{B}{K} = n$.
Picking a base $b_1, \dots, b_n$ of $B$ over $K$, we view $B$ as
the affine variety $\An$ of dimension $n$, identifying an element
$\sum_{i=1}^n \a _i b_i$ ($\a_i \in K$) with the vector $(\a_1,
\dots, \a_n)$. Usually $B$ is a $K$-algebra, but we formally do
not need this requirement.

Suppose $F$ is a subfield of $K$ and $V \! \subset\! B$ is a
vector space over $F$. The {\bf \Zcr} of $V$ inside $B$, denoted
by $\cl[B]{V}$, is the closure of $V$ inside $B$ via the Zariski
topology of $\An$ (identifying $B$ with $\An$). When $B$ is
understood, we write $\cl{V}$ for $\cl[B]{V}$.
\end{defn}

Recall that the Zariski topology of the affine variety $\An$ over
$K$ is defined as having its closed sets be precisely those sets
of simultaneous zeros of polynomials from the (commutative)
polynomial algebra $K[\la _1, \dots, \la _n]$. In other words, a
closed subspace of $B$ can be defined by (finitely many)
polynomials.

\begin{rem}\Label{polyf}
When we fix a base $b_1,\dots,b_n$ for $B$, any polynomial $f \in
K[\lam_1,\dots,\lam_n]$ can be viewed as a function $f \co B \ra
K$ by assigning $f(\a_1b_1+\cdots+\a_nb_n) = f(\a_1,\dots,\a_n)$.
\end{rem} %

A polynomial $f(\lam_1,\dots,\lam_n)$ is called a {\bf polynomial
relation} on $A$ if $f(A) = 0$, in the sense of Remark
\ref{polyf}. Thus a polynomial relation $f(\la_1, \dots, \la _n)$
 is always taken in $\le n$ indeterminates, and we check it by evaluating it on the
coordinates of a single element $a$, for each $a$ in $A$. In
contrast, in PI-theory, a polynomial identity $g(x_1, \dots, x_m)$
of $A$ (resp.\ of $B$) can be in any number of indeterminates,
specialized to $m$ elements of $A$ (resp.\ of $B$).

\begin{rem}\Label{indep}
 The \Zcr\  does not depend on the choice of base of
$B$ over~$K$, since a linear transformation induces an
automorphism of the polynomial ring (i.e., sends polynomial
relations to polynomial relations) and thus does not change the
Zariski topology.
\end{rem}

The \Zcr\ does depend on the way in which $V$ is embedded in $B$
as an $F$-space, even for $F$ infinite. In particular, for an
$F$-algebra $A$ contained in a $K$-algebra $B$, the notation
$\cl[B]{A}$ should also indicate the particular representation of
$A$ into $B$, as evidenced in the following example. (But
nevertheless, the representation is usually understood, and so is
not spelled out in the notation.)

\begin{exmpl}
For $F = \R$, $K = \C$, and $B = \M[n](\C)$, we could embed $A =
\mathbb C$ into $M_2(\C)$ as scalar matrices. On the other hand,
in the spirit of Example~\ref{basexa2}, we could identify $\C$
with $\set{\smat{\alpha}{0}{0}{\bar\alpha}: \alpha \in \C}$, where
$\bar{\phantom{w}}$ denotes the usual complex conjugation. In the
first case, the \Zcr\ of $A$ is $A$ itself, which is isomorphic to
$ \C$. In the second case, the \Zcr\ of $A$ is
$\smat{\C}{0}{0}{\C}\cong \C \times \C$, which has larger
dimension!

Although in this example $A\cong \C$ and thus $A$ is a
$\C$-algebra, it is not \Zcd\ in $M_2(\C)$. Thus, $\C$ need not be
\Zcd\ in $M_2(\C)$ as an $\R$-algebra. But note here that $A$ is
not a $\C$-subalgebra of $M_2(\C)$, and in fact we have the
following remark.
\end{exmpl}

\begin{rem}\Label{Kcl}
Any $K$-subspace $V$ of $B$ is \Zcd. (In particular, any
$K$-subalgebra of $B$ is \Zcd.) Indeed, a $K$-subspace is an
algebraic subvariety, defined by linear relations.
\end{rem}

In particular, we have:
\begin{lem}\Label{Kcl2}
$\cl{A} \subseteq KA$ inside $B$.
\end{lem}
\begin{proof}
We saw in Remark~\ref{Kcl} that $KA$ is \Zcd. Thus, the \Zcr\
$\cl{A}$ of $A$ is always contained in $KA$.
\end{proof}

Thus, we call $KA$ the {\bf \fcr} of $A$.

\begin{prop}\Label{finf}
If $F$ is infinite, then the \Zcr\ of an $F$-vector space $A$ is
equal to the \fcr\ of $A$.
\end{prop}
\begin{proof}
By definition, $\cl{A}$ is composed of the common zeros in $B$ of
the polynomial relations of $A$. Let $a \in A$, and let $f \in
K[\lam_1,\dots,\lam_n]$ be a polynomial relation. Then $f(\alpha
a) = 0$ for every $\alpha \in F$; viewing $\alpha$ generically, we
see that $f(\alpha a)$ is identically zero. Therefore, $f(\alpha
a) = 0$ for every $\alpha \in K$, which proves that $K a \sub
\cl{A}$.
\end{proof}

\begin{rem}\Label{interpol}
{\ }
\begin{enumerate}
\item If a vector space is \Zcd, then any subset
defined by polynomial relations is \Zcd.

\item \Label{interpoliii}
If \ $V\! \sub\! B_0 \!\sub\! B$, then the \Zcr\ of $V$ in $B$ is
equal to the \Zcr\  of $V$ in $B_0$. (Indeed, $B_0$ is closed in
$B$ by \Rref{Kcl}.)

\item Suppose $A_i\! \subseteq\! B_i$ for $i=1,2$, where $B = B_1 \oplus B_2$.
Then $$\cl[B]{(A_1\!+\!A_2)} = \cl[B_1]{A_1}+\cl[B_2]{A_2}.$$
(Indeed, $(b_1,b_2)\in B_1 \oplus B_2$ satisfies all the
polynomial relations of $A_1 + A_2$ iff the $b_i$ satisfy all
polynomial relations of $A_i$, for $i=1,2$.)
\end{enumerate}
\end{rem}

The distinction between finite and infinite fields, which is
crucial in what is to come, is explained by the following
observation.

\begin{exmpl}
\begin{enumerate}
\item  If $F$ is an infinite subfield of $K$, then $F$ satisfies
only the identities resulting from commutativity, and thus $\cl{F}
= K$ (this follows, \eg, from \Pref{finf} below). On the other
hand, if $F$ is a finite field of order $q$, then $\lam^q - \lam =
0$ is an identity and $\cl{F} = F$.

\item The \Zcr\  of $A = \smat{F}{K}{0}{F}$ in $M_2(K)$ is $A$ if $F$
is finite and $\smat{K}{K}{0}{K}$ otherwise.
\end{enumerate}
\end{exmpl}

\begin{exmpl}\Label{cap}
If $A_i$ are subsets of $B$, then clearly $\cl{(\bigcap A_i)} \sub
\bigcap(\cl{A_i})$. However, this may not be an equality. Indeed,
let $\mu$ be an indeterminate over $\F_q$, and take $A_i =
\F_p(\mu^i)$ for $i \in \N$, as subalgebras of the common
algebraic closure $K$. We have that $\cl{A_i} = K$ since these are
infinite fields, where $\bigcap A_i = \F_p$ which is closed.
\end{exmpl}

{}From now on, we assume that $B$ is a $K$-algebra.
\begin{thm}
{\ }
\begin{enumerate}
\item \Label{Si}
If $V$ is an $F$-subspace of $B$, then $\cl{V}$ is also an
$F$-subspace.
\item \Label{Sii}
If $A$ is an $F$-subalgebra of $B$, then $\cl{A}$ is also an
$F$-subalgebra.
\item \Label{Siii}
If $I $ is a left ideal of $ A$, then $\cl{I}$ is a left ideal of
$ \cl{A}$.
\item \Label{Siv}
If $I \triangleleft A$ then $\cl{I} \triangleleft \cl{A}$.
\end{enumerate}
\end{thm}
\begin{proof}
\begin{enumerate}
\item Given any $a \in B$ and any polynomial relation $f$
vanishing on $A$, define $f_a (x) = f(a+x)$. Clearly, for each
$a\in A$, $f_a$ vanishes on $A$, and thus on $\cl{A}$, i.e. $f(a +
r) = 0$ for all $r\in \cl{A}$. Thus, $f_r$ vanishes on $A$ for $r
\in \cl{A}$, implying $f_r$ vanishes on $\cl{A}$, i.e., $f(r+s) =
0$ for all $r,s \in \cl{A}$. This is true for every $f$ vanishing
on $A$, proving $r+s \in \cl{A}$; i.e., $\cl{A}$ is closed under
addition.

Likewise, defining $(\a f)(x) = f(\a x)$, we see for each $\a\in
F$ that $\a f$ vanishes on $A$ and thus on $\cl{A};$ i.e., $f(\a
r) = 0$ for all $r \in \cl{A}$, i.e., $\cl{A}$ is a $F$-vector
space.

\item Continuing the idea of (\ref{Si}), given any $a \in B$ and any
polynomial relation $f$ vanishing on~$A$, define $f_a (x) =
f(ax)$. Then, for each $a\in A$, $f_a$ vanishes on $A$ and thus on
$\cl{A}$, implying $f_a(r) = 0$ for all $r \in \cl{A}$. Repeating
this argument for $f_r$ shows that $f(rs) = 0$ for all $r,s \in
\cl{A}$, and we conclude that $rs \in \cl{A}$.

\item  By (\ref{Si}), $\cl{I}$ is a subgroup of $\cl{A}$. But for any
$a \in A$ and any polynomial relation $f$ vanishing on $I$, we
define $f_a(x) = f(ax)$, which also vanishes on $I$ and thus on
$\cl{I}$. Using the same trick and defining $f_r (x) = xr$, we
see, for any $r \in \cl{I}$, that $f_r$ vanishes on $A$ and thus
on $\cl{A}$, implying $ \cl{A} \cl{I}\subseteq \cl{I};$ i.e.,
$\cl{I}$ is a left ideal of $\cl{A}$.

\item Also apply the right-handed version of (\ref{Siii}).
\end{enumerate}
\end{proof}

The \Zcr\ acts functorially, and turns out to be a key tool in the
structure of algebras. To see this, we need to show that the \Zcr\
preserves various important structural properties. Sometimes it is
convenient to separate addition from multiplication in our
discussion. The \Zcr\  of an additive subgroup $(G,+)$ of
$\M[n](K)$ is a closed subgroup; i.e.,~an algebraic group.

\begin{prop}[{\cite[Cor.~7.4]{Hum}}]\Label{alggroup}
Suppose $G$ is any algebraic group that comes with a morphism of
algebraic groups $\psi\co G \to V$. Then $\psi(G)$ is \Zcd\ in
$V$.
\end{prop}

\begin{cor} Suppose $A$ is a \Zcd\ algebra and $\psi\co A
\to B'$ is a morphism of varieties. Then $\psi (A)$ is closed in
$B'$.
\end{cor}

\begin{cor}\Label{mapp}
For every $F$-subalgebra $A$ of $B$ and morphism $\psi \co B \ra
B'$, $\psi(\cl{A}) = \cl{\psi(A)}$.
\end{cor}
\begin{proof}
Since $\psi(\cl{A})$ is closed, we have that $\cl{\psi(A)} \sub
\psi(\cl{A})$; but $\psi(\cl{A}) \sub \cl{\psi(A)}$ by continuity
of $\psi$.
\end{proof}

Thus, we see how the power of algebraic group techniques enters
into the theory of \Zcd\ algebras. There is a newer theory of
algebraic semigroups \cite{Put} that would also enable us to
utilize the multiplicative structure; we return to this later.

\begin{cor}\Label{crucial}
Let $W \sub B$ be $K$-spaces. For any closed $F$-subspace $A \sub
B$, the factor space $A/(W\cap A)$ can be identified with a \Zcd\
subspace of $B/W$.
\end{cor}
\begin{proof}
Letting $\psi \co B \ra B/W$ be the projection morphism, $A/(W
\cap A) \isom (A+W)/W = \psi(A)$ is closed by \Cref{mapp}.
\end{proof}

\begin{cor}\Label{crucial1}
If $A$ is a \Zcd\ $F$-subalgebra of $B$ and $I \triangleleft B$,
then $A/(I \cap A)$ can be identified with a \Zcd\ subalgebra of
$B/I$.
\end{cor}
\begin{proof}
A special case of \Cref{crucial}.
\end{proof}

\subsection {PI's versus polynomial relations}

\begin{prop}\Label{notate}
The polynomial identities of the finite dimensional $K$-algebra
$B$ are determined by the polynomial relations in the Zariski
topology.
\end{prop}
\begin{proof}
Fixing the base $\set{b_i}$, we can take
any polynomial $f(x_1, \dots, x_m)$ defined on $B$, and, for any
$w_2, \dots, w_m \in B$, define $\hat f(x_1)$ via $\hat f(b) =
f(b, w_2, \dots, w_m)$. Writing $b$ ``generically'' as $\sum \la
_i b_i$ and $\hat f(b) = \sum \beta _k b_k$, we define $\hat f_k
(b) = \beta_k$. Putting each $\beta_k = 0$ in turn clearly
defines a polynomial relation, since multiplication of the base
elements of $B$ is given in terms of structure constants.

For example, suppose $b_ib_j = \sum \a _{ijk} b_k$ in $B$, and $f
= x_1x_2 - x_2 x_1$. Fixing $w_2 = \sum c_i b_i$, we have $$\hat f
(b) = \sum \la _i b_i \sum c_j b_j - \sum c_i b_i \sum \la _j b_j
= \sum _k \sum _{i,j} \a _{ijk}(c_j\la _i - c_i \la _j) b_k,$$ so,
for each $k$,
$$\hat f_k = \sum _{i,j} \a _{ijk}(c_j\la _i - c_i \la _j).$$
In this way, letting $w_2,\dots,w_m$ run over all elements of $B$,
we can view any polynomial identity as an (infinite) aggregate of
polynomial relations on the coefficients of the elements of $B$.
\end{proof}

The converse is one of our main objectives: {\emph{Can \Zcd\
algebras be differentiated by means of their polynomial
identities?}} For example, any proper $K$-subalgebra of $\M[n](K)$
satisfies the Capelli identity $c_{n^2}$, which is not an identity
of $\M[n](K)$ itself.

Although every multilinear identity of $A$ is also satisfied by
$KA$, we may have $\Var{A} \neq \Var{KA}$ in nonzero
characteristic. For example, if $A$ is the algebra of
Example~\ref{basexa2}.(\ref{BE2i}), then $x[y,z] = [y,z]x^{p^n}$
is an example of $A$ but not of $KA$. The pertinence of \Zcr\  to
PI-theory comes from the following obvious but crucial
observation.

\begin{lem}\Label{samevar}
$\VarF A = \VarF{\cl{A}}$.
\end{lem}
\begin{proof}
By~\Pref{notate}, any identity $f(x_1, \dots, x_m)$ of $A$
can be described in terms of polynomial relations. Thus the
polynomial identity $f$ passes to the \Zcr\  $\cl{A}$.
\end{proof}

 (The same proof shows that any generalized polynomial identity of
$A$ remains a generalized polynomial identity of $\cl{A}$;
likewise for rational identities.)

Let us first consider polynomial identities when $F$ is infinite.
Combining the lemma with \Pref{finf}, an $F$-subalgebra $A$ of $B$
is PI-equivalent to $KA$. Thus, up to PI-equivalence, when $F$ is
infinite, the \Zcd\ $F$-algebras correspond precisely to the
$K$-subalgebras of $\M[n](K)$, and we have nothing new.

On the other hand, nonisomorphic \Zcd\ algebras may be
PI-equivalent. For example, the algebra of diagonal matrices
$\smat{K}{0}{0}{K}$ is PI-equivalent to the algebra of scalar
matrices  $\left\{ \smat{\alpha}{0}{0}{\alpha}: \alpha \in K
\right\}$. Nevertheless, the \Zcr\ is a way of finding canonical
representatives of varieties of PI-algebras, which becomes much
more sensitive over finite fields.

\subsection{The structure of \Zcd\  algebras.}\Label{ss:struc}

As promised in the Introduction, we now show that \Zcd\ algebras
have a structure theory closely paralleling the structure of
finite dimensional algebras over an algebraically closed field.
Since we want to pass to the \Zcr\ in order to find a
``canonical'' algebra PI-equivalent to $A$, we want this to be
independent of the choice of $K$-algebra $B$ in which $A$ is
embedded. But presumably $A$ could be embedded in two $K$-algebras
$B_1$ and $B_2$, and could be \Zcd\ in $B_1$ but not in $B_2$. Towards
this end, we say $A$ is {\bf maximally \Zcd} if $A$ is \Zcd\ in
$B$, and every nonzero ideal of $B$ intersects $A$ nontrivially.

\begin{exmpl}
In general, a \Zcd\ algebra $A$ need not be maximally closed in~
$KA$. Indeed, let $A = \set{\smat{a}{0}{0}{a^p} \,:\, a\in K}$
where $p = \cha K$. Then $A$ is a field, but $KA =
\smat{K}{0}{0}{K}$ is not.
\end{exmpl}

However, we have the following useful fact:
\begin{prop}\Label{choice}
Every \Zcd\ $F$-subalgebra $A$ in $B$ is  maximally \Zcd\ with
respect to a suitable homomorphic image of $B$.

In particular, we may assume $A$ is maximally \Zcd\ in $KA$.
\end{prop}
\begin{proof}
We proceed by induction on $\dim_K B$. If $A$ is  not maximally
\Zcd, then there is some ideal $I$ of $B$ maximal with respect to
$I\cap A = 0$. But then $A \subseteq B/I$ by \Cref{crucial1}. The
second assertion follows by taking $B$ to be the $K$-space spanned
by $A$, a property retained by homomorphic images.
\end{proof}

For any subalgebra  $A$ of a matrix algebra $\M[n](K)$, every nil
ideal of $A$ is nilpotent, of nilpotence index bounded by $n$, by
a theorem of Wedderburn; \cf~\cite[Theorem 2.6.31]{Row2}. Thus,
there is a unique largest nil (and thus nilpotent) ideal of $A$,
which we write as $\Rad(A)$. Recall that $A$ is semiprime iff
$\Rad(A) = 0$.

\begin{prop}\Label{radcl}
$\Rad (\cl{A}) = \cl{\Rad (A)} = \cl{A} \cap \Rad(KA)$.
\end{prop}
\begin{proof}
$\Rad(A)$ satisfies the identity $x^n = 0$, which can be expressed
in terms of polynomial relations (\cf~\Pref{notate}); therefore
$\cl{\Rad(A)}$ is also nil. But clearly $\cl{\Rad(A)} \sub
\cl{A}$, so $\cl{\Rad (A)}\subseteq \Rad (\cl{A})$.

Likewise $\cl{\Rad(A)} \subseteq K\Rad(A)$, by
\Lref{Kcl2}, which in turn is a nilpotent ideal of $KA$
and thus contained in $\Rad(KA)$. This proves $\cl{\Rad(A)}
\subseteq \cl{A} \cap \Rad(KA)$. But the latter is a nilpotent
ideal of $\cl{A}$ so is in $\Rad(\cl{A})$, completing the circle
of inclusions. \end{proof}

The inclusion $K \Rad(A) \sub \Rad(KA)$ can in general be a proper
one. In fact, when $A$ is not maximally \Zcd\ in $B$, we can have
$\Rad(KA) \neq 0$ even if $A$ is simple.
\begin{exmpl}
Suppose $L/F$ is an inseparable field extension of dimension $p$,
viewed as an $F$-subalgebra of $\M[p](K)$, where $K$ is the
algebraic closure of $F$. Then $$KL \isom K[z\,|\,z^p = 0]$$ has
non-trivial radical.

Since $F$ is necessarily infinite, $\cl{L} = KL$.
\end{exmpl}

We are ready to turn to the \Zcr\ of factor images.

\begin{prop}\Label{radgood} Suppose $A$ is \Zcd\ in $B= KA$.
Then $A/\Rad(A)$ is \Zcd\ in $B/\Rad(B)$.\end{prop}
\begin{proof}
Let $J = \Rad (B)$. By \Cref{crucial}, $A/(A\cap J)$ is \Zcd\ in
$B/J$. But we are done, since $A\cap J = \Rad(A)$ by
Proposition~\ref{radcl}.
\end{proof}

\begin{prop} Suppose $A$ is \Zcd\ in $B$, and $z \in
\Cent{B}$. Then $A/\Ann_Az$ is \Zcd\ in $B/\Ann_Bz$.
\end{prop}
\begin{proof} $\Ann_A z = A \cap \Ann_B z$, so again we apply
\Cref{crucial}.
\end{proof}

\Zcd\  algebras behave strikingly similarly to finite dimensional
algebras  over an algebraically closed field.
\begin{prop}\Label{simpA}
If $\cl{A}$ is simple, then it is a matrix algebra, either over a
finite field or over the algebraically closed field $K$.
\end{prop}
\begin{proof}
As a PI-algebra, $\cl{A}$ is finite dimensional over its center
$F$, and thus $\cl{A} \cong M_t(D)$ for some finite dimensional
division algebra $D$.  If $F$ is infinite, then $A$ is finite
dimensional over the algebraically closed field $K$ by
\Pref{finf}, and thus $\cl{A} = M_t(K)$. On the other hand, if $F$
is finite, then $D=F$ by Wedderburn's \Tref{Wed3}, so $\cl{A}
\cong M_t(F)$.
\end{proof}

Our next goal is to obtain a Wedderburn decomposition for a \Zcd\
algebra, into radical and semisimple parts (when $F$ is finite, as
the other case is trivial). When describing intrinsic
ring-theoretic properties of a \Zcd\  algebra $A$, we do not refer
explicitly to $B$, and thus we choose $B$ as we wish. Usually we
take $B = KA$. Here is an example of this point of view.

\begin{lem}\Label{centcl}
If $A$ is \Zcd\ in the $K$-algebra $B$, then $\Cent{A}$ is \Zcd.
\end{lem}
\begin{proof}
An element of $A$ is in $\Cent{A}$ iff it commutes with a (finite)
base of $B$, so  we are done by~\Pref{notate}.
\end{proof}

\begin{prop}\Label{mat}
If $A$ is prime and \Zcd, then $A$ is a matrix algebra (either
over a finite field or over the algebraically closed field $K$).
\end{prop}
\begin{proof}
We choose $B=KA$. But then $\Cent{A}$ is a \Zcd\ domain and must
be either finite (and thus a field) or $K$ itself. Hence
$\Cent{A}$ is a field, so $A$ is a prime PI-algebra whose center
is a field, implying $A$ is simple, so we are done by
\Pref{simpA}.
\end{proof}

\begin{thm}\Label{semis}
Suppose $A$ is semiprime and \Zcd. Then $A$ is semisimple, namely
isomorphic to a direct product of matrix algebras over fields.
\end{thm}
\begin{proof}
By \Pref{choice}, we may assume $A$ is maximally closed in $B =
KA$.
But $\Rad (B)$ is a nilpotent ideal, so would intersect $A$ at a
nilpotent ideal, contrary to hypothesis unless $\Rad (B) = 0$.
Hence $B = S_1 \times \dots \times S_t$ is a direct product of
simple $K$-algebras. By \Cref{crucial}, the projection $A_i$ of
$A$ into $S_i$ is \Zcd. Furthermore, the  $A_i$ are prime, since
otherwise, taking nonzero ideals $I_1, I_2$ of $A_i$ with $I_1I_2
= 0$, we have $(I_1K)(I_2K) = 0$ in $S_i$, contrary to $S_i$ prime.
But then, by Proposition~\ref{mat}, $A_i$ is a matrix algebra over
a field. Writing $S_i = B/P_i$ for maximal ideals $P_i$ of $B$, we
have $A_i \approx A/(P_i\cap A)$, implying $P_i \cap A$ are
maximal ideals of $A$, with $\bigcap _{i=1}^t (P_i \cap A) = 0$.
Hence $A$ is semisimple.\end{proof}

\begin{cor}
If $A$ is \Zcd then $\Rad (A)$ is also the Jacobson   radical of $A$.
\end{cor}
\begin{proof} %
$A/\Rad (A)$ is semiprime, and thus semisimple, implying $ \Rad
(A)$ is also the Jacobson radical.
\end{proof}

Let us recall some technical ring-theoretic results from
\cite{Row2}. Any nil ideal  is idempotent-lifting, by
\cite[Corollary 1.1.28]{Row2}. An algebra $A$ is {\bf semiperfect}
when $\Rad (A)$ is nil and $A/\Rad (A)$ is semisimple, so we
instantly have the following result:

\begin{cor} %
Any  \Zcd\ algebra  $A$ is
semiperfect.%
\end{cor}

\begin{prop}
If $A$ is \Zcd, then so is $A/\Rad(A)$.
\end{prop}
\begin{proof}
We may assume $B = KA$ (\Rref{interpol}(\ref{interpoliii})) and
then apply Proposition~\ref{radgood}.
\end{proof}

We can now find an analog to the Krull-Schmidt theorem.

\begin{thm}\Label{KrullShexp2} If $A$ is \Zcd, then
there is a direct sum decomposition $A = \bigoplus_{i=1}^t Ae_i$ of
$A$ into indecomposable modules, and this decomposition is unique
up to isomorphism and permutation of components.\end{thm}
\begin{proof}
By \cite[Lemma 2.7.18]{Row2}, since $A$ is semiperfect.
\end{proof}

Our main structural result is an analog of Wedderburn's Principal
Theorem, \Tref{Wed2}, which played such a crucial role in Kemer's
proof of Specht's conjecture in characteristic $0$. This result is
the version that we need in characteristic $p$.

\begin{thm}\Label{Zarcl1} If $A = \cl{A}$, then $A$ has a Wedderburn
decomposition  $A = S\oplus J$, where $J = \Rad(A)$ and $S \cong
A/J$ is a subalgebra of $A$.\end{thm}
\begin{proof}
By Wedderburn's Principal Theorem
\cite[Theorem~2.5.37(Case~I)]{Row2} it is enough to prove $A/J$ is
split semisimple. But $A/J$ is \Zcd\, by \Pref{radgood}, so we are
done by \Tref{semis}.
\end{proof}

\subsection{Subdirect decompositions of \Zcd\  algebras}

\begin{rem}\Label{subirr} Suppose $B$ is a subdirect product of $K$-algebras
$B_1$ and $B_2$. Since annihilator ideals can be defined through
polynomial relations, a \Zcd\ subalgebra $A$ of $B$ is a subdirect
product of \Zcd\ algebras, and arguing by induction on
$\dimcol{B}{K}$, we may conclude that any \Zcd\ algebra $A$ is a
finite subdirect product of \Zcd\ algebras whose \fcr{}s are
subdirectly irreducible.

Of course, if $A$ is the subdirect product of $A_1, \dots, A_m$,
then $$\id (A) = \id (A_1 \times \cdots \times A_m) = \bigcap
_{i=1}^m \id (A_i),$$
thereby reducing the study of $\id(A)$ to
the subdirectly irreducible case.
\end{rem}

Let us summarize what we have done so far and indicate what is
still missing. Suppose $A$ is \Zcd\ with \fcr\ $B$. By
Wedderburn's Principal Theorem, \Tref{Wed2}, we can write $B= S
\oplus J$, where $S$ is the semisimple part and $J$ is the radical
part, and we may assume that $B$ is subdirectly irreducible and
has block triangular form, so $A$ involves the same nonzero
components. Furthermore, $A/J$ is semisimple and thus a direct sum
of central simple algebras. We can write $A$ in upper triangular
form. On the other hand, we do not yet have a good description of
the relations among the components; these are treated in the next
section, in particular Theorem~\ref{linear}.

\section {Types of polynomial relations of a \Zcd\
algebra}\Label{sec:3}

As before, we study an $F$-algebra $A$ contained in an
$n$-dimensional $K$-algebra $B$, where $F \sub K$ are fields.
Since PI-theory deals so extensively with the $T$-ideal $\id (A)$
of identities of an algebra $A$, it is reasonable to expect that
the ideal of polynomial relations will play an important role in
our analysis of \Zcd\ algebras in $B$.
\begin{defn}
For an $F$-algebra $A$ contained in a $K$-algebra $B$ with basis
$b_1,\dots,b_n$, $\pol(A) \normali K[\la_1, \dots, \la _n]$ is
defined as $\pol(A) = \set{f \suchthat f(A) = 0}$.

\end{defn}

Our next objective is to find the ``best'' generators of the
polynomial relations. This is a major issue, taking much of the
remainder of this paper.

Unlike the \Zcr\ (\cf~\Rref{indep}), $\pol(A)$ does depend on the
choice of a base for $B$. In fact, the general linear group
$\GL[n](K)$ acts on bases of $B$ by linear transformations and on
polynomials (and ideals of polynomials) by left composition. {}From this point of view, relations are studied up to the action of $\GL[n](K)$
on ideals of $K[\lam_1,\dots,\lam_n]$, and we may simplify the
relations by proper choice of the base.

In some ways, although polynomial relations generalize polynomial
identities, their ideals are easier to study than $T$-ideals,
since we view them in a much more manageable algebra, the
commutative algebra $K[\la_1, \dots, \la _n]$.

\begin{rem}\Label{rem0}
Some initial remarks in studying $\pol (A)$ are as follows:

\begin{enumerate}
\item \Label{rem0i}
We may assume $A$ has some element $a = \sum \a _i b_i$ with
$\a _n \ne 0;$ otherwise we have the polynomial relation $\la_n$
that we can use to eliminate all monomials which include $\a _n$.

\item Since $A$ is a group under addition, we know that $0 \in A$,
so $f(0) = 0$ for all polynomial relations $f$ of $A$. In other
words, the only polynomials in $K[\la_1, \dots, \la _n]$ that we
consider are those having  constant term $0$.
\end{enumerate}
\end{rem}

\begin{rem}\Label{finmany}
$\pol (A)$, being an ideal of the Noetherian ring $K[\la_1, \dots,
\la _n]$, is finitely generated. Thus all the polynomial relations
of $A$ are consequences of finitely many polynomial relations.
\end{rem}

{}Thus,  Specht's problem becomes trivial for polynomial
relations. For example,  the matrix algebra  $M_n(F)$, viewed as an
$n^2$-dimensional affine space  over the field $F$, satisfies the
polynomial relations $\la_i^q-\la_i$ iff $F$ is finite and
satisfies the identity $x^q -x$.

\subsection{Additively closed Zariski-closed sets}

Next, we adapt the well-known theory of multilinearization. Since
this uses the additive structure, we focus the next investigation
to this case.

\begin{rem}{\ }
\begin{enumerate}
\item An additive group $A \sub K^{(n)}$ acts on its set of relations $\pol(A)$
via translation: $a \co f(\lam) \mapsto f(\lam+a)$.
\item If $A$ is an $F$-space, $\mul{F}$ acts on $\pol(A)$ via
scaling: $\alpha \co f(\lam) \mapsto f(\alpha \lam)$.
\end{enumerate}
\end{rem}

\begin{defn} A polynomial $f(\la_1, \dots, \la _n)\in K[\la_1, \dots, \la _n]$ is
{\bf quasi-linear} (with respect to $A$) if $$f(\la +a)=f(\la)+
f(a), \qquad \forall a \in A.$$

 Also,   $f$ is {\bf
$F$-homogeneous}  if there is $d\in \N ^+$ such that, for each $\a
\in F$,
$$f(\a \la_1, \dots, \a  \la _n) =
 \a ^d  f(\la_1, \dots,   \la _n).$$
\end{defn}

\begin{rem}
{\ }
\begin{enumerate}
\item A quasi-linear polynomial (with respect to any $A$)
necessarily has zero constant term, for $f(\lam) = f(\lam + 0) =
f(\lam)+f(0)$. \item Over an infinite field, the quasi-linear
polynomials  are linear. However, note that $x^p-x$ is a
non-linear but quasi-linear $\F_p$-homogeneous polynomial.
\end{enumerate}
\end{rem}

\begin{prop}\Label{TR}
{\ }
\begin{enumerate}
\item Suppose $A$ is an additive group. The ideal of polynomial
relations of $A$ is generated by quasi-linear polynomial
relations.

\item \Label{TRii} If $A$ is an $F$-vector space, the ideal of
polynomial relations of $A$ is generated by  quasi-linear
$F$-homogeneous polynomial relations.\end{enumerate}
\end{prop}
\begin{proof}
\begin{enumerate}
\item  Suppose $f\in \pol(A)$. Given any $a \in A$, we define the
new polynomial relation $\Delta_a f(\la) = f(\la +a)-f(\la)$. This
has smaller degree than $f$, and clearly is a consequence of $f$.

On the other hand, $f$ is a formal consequence of $\{\Delta_a f :
a \in A \}$. Indeed, since   $f$ has constant term $0$ (by
Remark~\ref{rem0}), we have $f(0) = 0$, and thus $$f(a) = f(a) -
f(0) =  \Delta_a f(0).$$ If $\Delta_a f(A) = 0$ for all $a$, this
implies $f(a) = 0$, so that $f(\la)$ is a polynomial relation of
$A$.

We can thus replace $f$ by finitely many $\Delta_a f$, in view of
Remark~\ref{finmany}, and repeating this process, eventually we
get $\Delta_af(\la) = 0$ for all $a \in A$, i.e., $f$ is
quasi-linear.

\item Given $\a_i \in F_i$, we can define $\triangledown f =
f(\a\la_1, \dots, \a \la _n) - \a^d f(\la_1, \dots,   \la _n)$,
where $d$ is the (total) degree of $f$. This provides a polynomial
relation with fewer nonzero monomials, as is
$f-\gamma\triangledown f$ (for suitable $\gamma \in K$, provided
$\triangledown f \ne 0)$. On the other hand, $f =
\gamma\triangledown f + (f-\gamma\triangledown f)$, so we continue
by induction,  unless $\triangledown f = 0$. But this means that
$f$ is $F$-homogeneous.
\end{enumerate}
\end{proof}

\begin{rem}\Label{add} {\ }
\begin{enumerate}
\item For an additive subgroup $V \sub K^{(n)}$ and a polynomial $f$
quasi-linear with respect to $V$, the intersection $V \cap Z(f)$
is a group, where $$Z(f) = \set{c \in K^{(n)} \,:\, f(c) = 0}$$ is
the variety associated to $f$. Indeed, if $f(a) = f(b) = 0$, then
$f(a+b) = 0$, and $f(-a) + f(a) = f(-a+a) = f(0)$, implying $f(-a)
= 0$.

\item  In particular, if $f$ is quasi-linear with respect to $K^{(n)}$,
then its variety is a group.

\item  The variety of an arbitrary quasi-linear $F$-homogeneous
polynomial relation $f$ is a vector space over $F$, since if $f(a)
=  0$, then $$f(\a a) = f(\a a_1, \dots, \a  a _n) =
 \a^{d} f(a_1, \dots,   a _n)=0.$$
\end{enumerate}
\end{rem}

Having reduced to quasi-linear (perhaps also $F$-homogeneous)
polynomial relations, we would like to determine their form.

\begin{defn}\Label{FTdef}
Suppose $\cha F = p$. Let $q = \card{F}$, setting $q = 1$ if $F$
is infinite.

\begin{enumerate}
\item\Label{FTi} A polynomial relation $f\in K[\la_1, \dots, \la
_n]$ is of {\bf weak Frobenius type} if $f$ has the following
form: \begin{equation}\label{Frobtyp}\sum _{i=1}^n \sum _{j\ge 1}
c_{ij} \lam_i^{q_{ij}} = 0,\end{equation} 
where $c_{ij} \in K$ and
each $q_{ij}$ is a $p$-power. (Recall that our polynomial
relations have constant term $0$.)

\item \Label{FTii} The polynomial relation $f$ is a \defin{weak
$F$-Frobenius type} (also known as a \defin{$q$-polynomial} in the
literature) if, in (\ref{FTi}), we may take each $q_{ij}$ to be a
power of $q$.
\end{enumerate}
\end{defn}

Note that weak $F$-Frobenius type (resp.~ weak Frobenius type)
reduces to the linear polynomial relation $\sum c_i \lam_i = 0$
for $F$ infinite (resp.~in characteristic $0$). In view of
Remark~\ref{add}, the next result (which strengthens Proposition
\ref{TR}) characterizes algebraic varieties that are Abelian
groups.

\begin{thm}\Label{linear}
{\ }
\begin{enumerate}
\item\Label{Li}
The ideal of polynomial relations of an additive group $A$ is
generated by polynomial relations of weak Frobenius type.
Specifically, any polynomial relation is a consequence of finitely
many polynomial relations of weak Frobenius type.

\item\Label{Lii} The ideal of polynomial relations of an
$F$-vector space $A$ is generated by polynomial relations of weak
$F$-Frobenius type. Specifically, any  quasi-linear
$F$-homogeneous polynomial relation is a consequence of finitely
many  polynomial relations of weak $F$-Frobenius type.
\end{enumerate}
\end{thm}
\begin{proof}
\begin{enumerate}
\item  It is enough to prove the second assertion. We write a
polynomial relation $f= \sum h_{{d_1, \dots, d_n}}$, where $h$ is
the monomial with multi-degree ${d_1, \dots, d_n}$, i.e., of the
form $c\la _1 ^{d_1}\cdots \la _n ^{d_n}$. Clearly $\Delta _a f=
\sum \Delta _a h_{{d_1, \dots, d_n}}$, so we consider a typical
monomial $$ h_{{d_1, \dots, d_n}} = c\la _1 ^{d_1}\cdots \la _n
^{d_n}.$$

Taking $a = \sum \a_i b_i$ with $\a_n\ne 0$
(\cf~Remark~\ref{rem0}(\ref{rem0i})), we have
$$\Delta _a (h) = c(\la _1+ \a_1) ^{d_1}\cdots (\la _n + \a
_n)^{d_n} -c\la _1 ^{d_1}\cdots \la _n ^{d_n},$$ so the highest
monomial not cancelled (under the lexicographic order giving
highest weight to $\la _1$) is $$d_n\a _n c \la _1 ^{d_1}\cdots
\la _n ^{d_n -1}.$$
But this must be $0$, so $  d_n\a _n c $ must be $0$ in $K$, i.e.
$p\,|\,d_n$, where $p = \cha (K)$. Continuing in this vein, we see
that the highest term in $\Delta _a (h)$ is
$$\a _n ^q c\binom {d_n}{q}\la _1
^{d_1}\cdots \la _n ^{d_n -q},$$ for some $p$-power $q$. This is a
contradiction unless it is cancelled by $\Delta_a(h')$ for some
other monomial $$h' = c'\la _1 ^{d'_1}\cdots \la _n ^{d'_n}.$$ By
maximality assumption on the degrees, we must have $d_i' = d_i$
for all $i\le {n-1}$, and $d'_n = d+n+q'$ for some $p$-power
$q'$. Then $\Delta_a (h')$ contains the term
$$\a _n ^{q'} c'\binom {d_n}{q'}\la _1
^{d_1}\cdots \la _n ^{d_n -q'}.$$ %
Perhaps other terms of this form come from other monomials, but
the upshot is that there is a linear combination of $p$-powers of
$a_n$ that are $0$. But this is true for each $a_n$, and thus
yields a polynomial relation $g(\la _n)$. Applying $\Delta_{\a}$
to $g(\la_n)$ for every $\a \in K$ enables us to reduce the power,
unless $\Delta _\a(g(\la_n)) = 0$ for all $\a$, \ie, $g$ is
quasi-linear. But in this case we can add $g$ to our list of
polynomial relations, and use $g$ to reduce the degree of $f$ in
$\la _n$.

Thus one continues until $\la _n$ does not appear in the highest
monomial of $f$. Applying the same argument whenever a monomial
has at least  two indeterminates in it, we eventually reach the
situation in which each monomial has a single indeterminate, \ie,
$h_i = \sum_i c_{ij} \la _i^{d_{ij}}$. Applying $\Delta$ lowers
the degree unless every $d_{ij}$ is a $p$-power, as desired.

\item Continuing (\ref{Li}), applying $\triangledown$ (as defined in
the proof of \Pref{TR}(\ref{TRii})), lowers the degree unless
every $d_{ij}$ is a $q$-power, as desired.
\end{enumerate}
\end{proof}

The claim of Example~\ref{basexa2}(\ref{BE2ii}) follows as an
immediate consequence, for if the algebra satisfied any extra
identity its corresponding polynomial relations must come from the
fact that $\s$ was Frobenius.

\subsection{Multiplicatively closed Zariski-closed
sets}

Our next theorem is not needed for our exposition, since we never
deal with multiplicatively closed subvarieties of $K^{(n)}$ unless
they are algebras. Nevertheless, the result is interesting in its
own right and complements the other results.
\begin{exmpl} The subvariety  $(K \times \{ 0\}) \cup (\{
0 \} \times K)$ of $K^{(2)}$ is defined by the polynomial
relation~$\la_1\la_2$.
\end{exmpl}

\begin{thm} %
Suppose $A$ is a \Zcd\ (multiplicative) submonoid of $K ^{(n)}$. Then the
polynomial relations of $A$ are generated by polynomial relations
of the form $$\la _{i_1}\cdots \la _{i_t} = 0; \qquad \la
_1^{i_1}\cdots \la _n ^{i_n}=\la _1^{j_1}\cdots \la _n ^{j_n}\quad
\text{for}\quad i_1, \dots, i_n ,j_1, \dots, j_n \in \Z.$$
\end{thm}
\begin{proof} To simplify notation, we write ${\mathbf i}$ for $i_1,
\dots, i_n$, $\la^{\mathbf i}$ for $\la _1^{i_1}\cdots \la _n
^{i_n}$, and ${\mathbf \a}^{\mathbf i}$ for $\a _1^{i_1}\cdots \a
_n ^{i_n}$. On the other hand, ${\mathbf \a}^m\la^{\mathbf i}$
designates $\a_1^m\la _1^{i_1} \cdots \a _n ^m\la _n^{i_n}$.

Take any polynomial relation $f = \sum c_{\mathbf i} \la
_1^{i_1}\cdots \la _n ^{i_n} = \sum_{\mathbf i} c_{\mathbf i}\la ^
{\mathbf i}$. By definition, $f(\mathbf \a) = 0$ for any $\mathbf
\a \in A$, and thus $f(\mathbf \a^j) = 0$ for each $j$, since $A$
is assumed multiplicative.

Cancelling out any $\la _i$ appearing in a polynomial relation
$\la _i = 0$, we induct on the number of indeterminates in $f$
and then on the number of monomials of $f$. Take any point $(\a_1,
\dots, \a _n)$. For $\gamma \in F$, we write $f_\gamma$ for the
sum of those monomials $c_{\mathbf i}\la^{\mathbf i} $ for which
$\mathbf \a ^ {\mathbf i} = \gamma$. Then by definition, $f = \sum
f_\gamma$. But $$f(\a^m \la) = \sum c_i \a_1^{mi_1}\cdots \a _n
^{mi_n} \la _1^{i_1}\cdots \la _n ^{i_n}= \sum \gamma ^m f_\gamma
(\la),$$ so by a Vandermonde argument, we  see that each
$f_\gamma$ is a polynomial relation of $A$.

Thus, one can separate $f$ into a sum of polynomial relations
involving fewer monomials (and conclude by induction) unless
(comparing monomial by monomial) all the $\a_{\mathbf i}$ are
equal.

 But this means, for each monomial $\la _1^{i_1}\cdots \la _n ^{i_n}$
 and $\la _1^{j_1}\cdots \la _n ^{j_n}$ that $A$ satisfies the
equalities
\begin{equation} \Label{mon1}\a _1^{i_1}\cdots \a _n ^{i_n}=\a _1^{j_1}\cdots
\a _n ^{j_n}
\end{equation}
for all $(\a _1, \dots, \a _n) \in A$, so that
$$\la _1^{i_1}\cdots \la _n ^{i_n}=\la _1^{j_1}\cdots \la _n ^{j_n}$$
is a polynomial relation of $A$. In other words, $\a = \a
_1^{i_1}\cdots \a _n ^{i_n}$ is independent of the choice of the
monomial $\la _1^{i_1}\cdots \la _n ^{i_n}$ of $f$, so $\a^m f$ is
an identity. But now, working backwards,
$$0 = f(\a_1, \dots, \a _n) = \a f(1,\dots, 1),$$
implying $\sum c_i = f(1,\dots, 1) = 0$. This implies that $$f =
\sum _i  c_i ( \la _1^{i_1}\cdots \la _n ^{i_n}- \la_1^{j_1}\cdots
\la_n ^{j_n})$$ is a  consequence of the relations~\eqref{mon1}.
\end{proof}

\subsection{Polynomial relations of commutative algebras}

Now let us utilize the fact that $A$ is an $F$-algebra.

\begin{defn}
A polynomial relation $f$ is of {\bf Frobenius type} if it has one
of the following three forms, where $p = \cha (F)$:

(i) $\la _i = 0$,

(ii) $\la _i- \la _i^s = 0$, where $s$ is a $p$-power, or

(iii) $\la _i - \la _j^s = 0$, $j \ne i$, where $s$ is a
$p$-power.

The polynomial relation $f$ has {\bf $F$-Frobenius type} if, in
(ii) and (iii), $s$ is a $q$-power, $q = \card{F}$ (where, as
usual, we put $q = 1$ if $F$ is infinite).
\end{defn}

\begin{lem}\Label{glue0} {\ }
Suppose $A$ is an additive subgroup of $K^{(n)}$, defined by
polynomial relations of the form $\la _i = 0$ and $\la _i^{q_i} =
\la _j^{q_j}$ for natural numbers $q_i$ and $q_j$. Then any such
relation is equivalent to a polynomial relation of Frobenius type
(ii) or (iii).\end{lem}
\begin{proof}
First we discard all components $i$ for which $\la _i = 0$ holds. Next, assuming
$q_i \le q_j$, one could then factor out $x^{q_j -q_i}$ (since $K$
is a field) to get $\la _i - \la _j^q$ for some $q\ge 1$. We are
done if $q=1$, so assume $q>1$. This relation holds for $\la _i =
\la _j = 1$, so additivity of $A$ gives the relation
$$(\la _i+1)- (\la _j+1)^q = (\la _1-\la _j^q) + \sum _{\ell = 1}^{q-1}\binom q \ell
\la _j^\ell.$$

If some $\binom q \ell \ne 0$, this translates to algebraicity of
the $j$ component of $A$, which must thus be defined in a finite
subfield of $K$, say of dimension $m$ over $\F_p$, and, as is well
known, every polynomial in $x$ is satisfied by the field or is a
multiple of $x^{p^m}-x$, and we have reduced to type (ii).

Thus we may assume that $\binom q \ell = 0$ for all $1 \le \ell
<q$. But this clearly implies $K$ has positive characteristic
$p>0$, and therefore $q$ is a power of $p$ (since otherwise, if
$q'$ is the highest power of $p$ less than $q$, then $\binom q{q'}
\ne 0$ in $K$.)

We may now assume $f = \la_i - \la _j ^s$. Write $s = q^jt$ for
$t$ prime to $q$. Then $f$ reduces to the polynomial relation,
$\la _i - \la _j ^t$. But taking $\a \in F$ with $\a ^t \ne \a$,
we see that $$\a^t (\la _i - \la _j ^t)- (\a\la _i - \a^t\la _j
^t) = (\a^t -\a \la _i),$$ yielding the polynomial relation $\la
_i$ and thus also $\la _j$.
\end{proof}

\begin{thm}\Label{comZar}
Suppose $A$ is a  commutative, semiprime \Zcd\ $F$-subalgebra of a
finite dimensional commutative $K$-algebra $B$. Then $\pol(A)$ is
generated by finitely many polynomial relations of Frobenius type.
\end{thm}

\begin{proof}
We can write $A \subseteq K_1 \times \dots \times K_t$, where each
$K_i \approx K$. By \Tref{linear}, it is enough to consider
polynomial relations of weak $F$-Frobenius type, i.e. of the form
$f = \sum _{i=1}^n \sum _{j\ge 1} c_{ij} \la_i^{q_{ij}}$. If we
have any relations of type (i), we can simply remove all terms
with $\la _i$, and ignore~$\la _i$.

The constant term of $f$ is $0$, where the $q_{ij}$ are powers of
$p$. If all the $q_{ij}$ are divisible by~$p$, then we may take
the $p^{\mbox{th}}$ root and still have a polynomial relation, so we may
assume that some monomial of $f$ is linear. For convenience we
assume $f$ has a monomial linear in $\la _1$.

If $A$ satisfies a polynomial relation only involving $\la _i$,
this means that the projection of $A$ onto $K_i$ is a finite field
$F_i$, which, if nontrivial, satisfies some identity $g_i = \la
_i^q - \la _i$, where we take $q$ minimal possible. But it is
well known that every polynomial satisfied by all elements of
$F_i$ is a consequence of $g_i$, so we may assume that all
polynomial relations involving only $\la _i$ are a consequence of
$g_i$.

We claim that modulo the $g_i$, either $f$ becomes $0$ or $f$
yields some polynomial relation of Frobenius type (iii). Suppose
$\la _1^i$ appears with two differing degrees. Write $f = \sum f_i
\la^i$. If each $f_i$ is a polynomial relation, then we continue
inductively. Otherwise take some nonzero value and conclude that
$F_1$ is algebraic of bounded degree over $F$, and thus is finite,
yielding a polynomial relation of type (ii), which thus is a
consequence of $g_i$.

Thus we may assume $\la_1$ appears in a single monomial. But this
means $f$ has the form $\la_1 + \sum c_j \la _j^{q_j}$, so
$$\a_1\la_1 + \a_j^{q_j}\sum c_j
\la _j^{q_j} = \a_1(\la_1 + \sum c_j \la _j^{q_j})$$ formally, or
in other words $\a _1 = \a _j^{q_j}$ for each $j$ appearing in
$f$, as desired.
 \end{proof}


\begin{rem}\Label{gluequiv}
We can combine Theorems \ref{linear} and \ref{comZar} for any
subalgebra $A$ of $B= \M[n](K)$, as follows: Let $A^1 = A \cup
\sum Fe_{ii} \subseteq B$. We define a relation on $\{1, \dots,
n\}$ by saying $i \equiv j$ if there is a nontrivial Frobenius
polynomial relation involving both $i$ and $j$, and we extend it by
transitivity to an  equivalence relation on $I = \{1, \dots, n\}$.
If $I_u$ is some equivalence class, then $A^1$ contains some
element
$$e_u= \sum _{i \in I_u} \a_i e_{ii}.$$ But then, for any $a\in
A$ and any $u,v \in I$, clearly $e_u a e_v \in A^1$, and the only
indices appearing nontrivially are from $I_u \times I_v$. We call
a quasi-Frobenius polynomial relation {\bf basic} if it only
involves coefficients from $I_u \times I_v$ for suitable
equivalence classes of $I$. In this way, we see that any
quasi-Frobenius polynomial relation reduces to the sum of basic
quasi-Frobenius polynomial relations.
\end{rem}

\section{Explicit representations of \Zcd\  algebras}\Label{sec:4}

We have seen in Theorem \ref{semis} that any semiprime \Zcd\
algebra is a direct sum of matrix components and thus has a very
easy representation inside $\M[n](K)$ along diagonal matrix
blocks. In order to describe the structure of a \Zcd\ algebra $A$
with nonzero nilradical $J$, we consider a faithful representation
of $A$ in a matrix algebra $\M[n](K)$. Throughout, we use this
particular representation to view $A \subseteq \M[n](K)$. Our
object is to find a ``canonical'' set of matrix units of
$\M[n](K)$ with which to view a \Zcd\ algebra $A$. The underlying
idea, introduced by Lewin \cite{Lew74} for PIs and utilized to
great effect in characteristic 0 by Giambruno and Zaicev
\cite{GZ}, is to write the algebra in something like upper
triangular form, in order to understand the placement of radical
substitutions in polynomial identities.

\begin{exmpl}
$A = \left(\begin{matrix} K & K & 0 \\ 0 & K & 0 \\ 0 & K & K
\end{matrix}\right)$.
Here $A/J$ can be identified with $\left(\begin{matrix} K & 0& 0
\\ 0 & K & 0 \\ 0 & 0& K \end{matrix}\right)$, and $J$ with
$\left(\begin{matrix} 0& K & 0 \\ 0 & 0 & 0 \\ 0 & K &
0 \end{matrix}\right)$. We can put $A$ into upper triangular form
by switching the second and third rows and columns to get
$\left(\begin{matrix} K & 0& K
\\ 0 & K & K \\ 0 & 0& K \end{matrix}\right)$.
\end{exmpl}

Unfortunately, we may not be able to straighten out $A$ so easily,
even in characteristic~$0$.

\begin{exmpl}\Label{glue00}
Let $A = \left\{\left(\begin{matrix}  a & b \\ 0 & a
\end{matrix}\right): \ a, b \in K\right\}$. This can also be viewed as the
$2\times 2$ matrix representation of the commutative algebra of
dual numbers of $K$, i.e., $\left(\begin{matrix} a & b \\ 0 & a
\end{matrix}\right)$ is identified with $a+ b \delta$, where
$\delta ^2 = 0$.
\end{exmpl}

In order to represent this as a triangular matrix ring, we must
identify certain components. Our objective in this section is to
describe how the identifications work for a particular
representation of a \Zcd\ algebra. Let us start with an easy
example that may lower our expectations.

\begin{exmpl}\Label{unglued} $F = \F_2(\mu)$, where $\mu$ is an
indeterminate over the field $\F_2$ of two elements
and $K$ is its algebraic closure. Then $F$ can be represented in
$K\times K$ by $a \mapsto (a^2,a^4)$, so the identification among
components is the relation $\la _2 = \la _1^2$.
\end{exmpl}

A direct identification between two components, via a polynomial
relation, is called {\bf gluing}. When components are not glued,
we say they are {\bf separated}.  In this paper, gluing is
considered mostly along the diagonal blocks, since off-diagonal
gluing turns out to be more complicated. As the above example
shows, gluing need not be ``onto'', when taken over infinite
fields.

\begin{defn}\Label{wbf}
Suppose $A$ is a \Zcd\ subalgebra of $\M[n](K)$ with radical $J$,
such that $A/J = A_1 \times \dots \times A_k$ with $A_u \cong
\M[n_u](F_u)$ for subfields $F_u \subseteq K$ ($u = 1,\dots,k$).

We say $A$ is in {\bf Wedderburn block form} if $n = \sum t(u)
n_u$, and for each $u$ there are $t(u)$ distinct matrix blocks
$A_u^{(1)}, \dots, A_u^{(t(u))}$ of size $n_u \times n_u$ along
the diagonal, each isomorphic to $A_u$, such that the given
representation $\varphi \co A \to \M[n](K)$ restricts to an
embedding $\varphi_u \co A_u \to A_u^{(1)}\times \dots \times
A_u^{(t(u))}$, where the projection $\varphi_u^{(\ell)} \co A_u
\to A_u^{(\ell)}$ is an isomorphism for each $1\leq \ell\leq
t(u)$. Furthermore, $J$ is embedded into strictly upper
triangular blocks (above the diagonal blocks). For each $u$, the
blocks $A_u^{(1)},\dots,A_u^{(t(u))}$ are {\bf glued} and belong
to the same {\bf gluing component}. For further reference, we
define $m = \sum t(u)$, the total number of diagonal blocks
(before gluing) in the representation of $A$.
\end{defn}

For algebras with $1$,  where $1$ is represented as the identity
matrix, obviously each diagonal Wedderburn block is nonempty.
However, for algebras without $1$, one could have some $B_u$
consisting only of $0$ matrices. In this case we say the block
$B_u$ is {\bf empty}.

Note that the glued blocks do not have to occur consecutively; for
example the semisimple part could be embedded in
$$\left(\begin{matrix}
A_1^{(1)} & 0& 0 & 0 & 0
\\ 0 & A_2^{(1)} & 0 & 0 & 0 \\ 0 & 0 & A_1^{(2)} & 0 & 0
\\0 & 0& 0  & A_1^{(3)}& 0\\0 & 0& 0   & 0 &A_2^{(2)}\end{matrix}
\right).$$

The radical belongs to blocks above the diagonal; for example,
$$\left(\begin{matrix} A_1^{(1)} & 0& 0 & J & J
\\ 0 & A_2^{(1)} & J & 0 & J \\ 0 & 0 & A_1^{(2)} & 0 & 0
\\0 & 0& 0  & A_1^{(3)}& 0\\0 & 0& 0   & 0 &A_2^{(2)}\end{matrix}
\right).$$

\begin{exmpl}\Label{GZB}
The following  basic illustration of Wedderburn decomposition,
without gluing, appears as the ``minimal algebra'' in Giambruno
and Zaicev \cite[Chapter 8]{GZ}, which they realize as an upper
block-triangular algebra
$$\(\begin{array}{cccc}
M_{n_1}(F) &  * & * & * \\
0 & M_{n_2}(F) & * & * \\
0 & 0 & \ddots & *  \\
0 & 0 & 0 & M_{n_t}(F)
\end{array}\).$$

The Giambruno-Zaicev algebra could be thought of as the
algebra-theoretic analog of the Borel subgroups of
$\operatorname{GL}(n,F)$. The semisimple part $S$ is the direct
sum $\bigoplus M_{n_i}(F)$ of the diagonal blocks, and the radical is
the part above these blocks, designated above as $(*)$.

 This kind of algebra first
came up in a theorem of Lewin \cite{Lew74}, who showed that any
PI-algebra $A$ with ideals $I_1,I_2$ satisfying $I_1I_2 = 0$ can
be embedded into an algebra of the form $\(\begin{array}{cccc}
A/I_1 &  * \\
0 & A/I_2
\end{array}\)$.

Giambruno and Zaicev proved, in characteristic $0$,  that for any
variety $\mathcal V$ of PI-algebras, its exponent $d$ (a concept
defined in terms of the asymptotics of the codimensions of
$\mathcal V$) is an integer and can be realized in terms of one of
these Giambruno-Zaicev algebras, as $d = \sum _{u=1}^k n_u^2$.
\end{exmpl}

\begin{rem}\Label{GZB1}
Belov \cite{Bel2} proved a parallel result to Giambruno-Zaicev's
theorem, for Gel'fand-Kirillov dimension in any characteristic.
Namely the GK-dimension of the ``generic'' upper
block-triangular algebra generated by $m$ generic elements is $$k
+ (m-1)\sum _{u=1}^k n_u^2.$$

In the same paper, Belov proved the following result: Suppose $A =
S\oplus J$ is the Wedderburn decomposition of a f.d.~algebra $A$,
with $S = A_1 \oplus \dots \oplus A_k$ (where $A_u$ are the simple
components). If there exist $x_u \in J$ such that $A_1 x_1A_2 x_2
\cdots x_m A_m \ne 0$ for some $m \le k$, then $A$ contains a
subalgebra isomorphic to the Giambruno-Zaicev algebra built up
from from $A_1, \dots, A_k$.
\end{rem}

Wedderburn block form refines Wedderburn's principal theorem.
Indeed, it is apparent by inspection that the part of $A$ along
the diagonal blocks is the semisimple part of $A$, and the part on
the blocks above the diagonal is the radical part. Note that this
example has not described the identifications among the
$A_u^{(\ell)}$. Clearly there must be gluing whenever some
$t(u)>1$, since $\dim (A_u^{(1)}\times \dots \times A_u^{(t(u))})
= t(u) \dim A_u$.

We can tighten these observations with some care. We start by
assuming that $A$ is a $K$-algebra ($K$ is algebraically closed,
as always). Then each $F_u = K$. An easy application of an
argument of Jacobson (spelled out in \cite[Theorem
 25C.18]{Row3}) yields:

\begin{thm}\Label{block1}
For $K$ algebraically closed, any finite dimensional $K$-algebra
$A$ can be put into Wedderburn block form.
\end{thm}

\begin{cor}\Label{block2}
Any \Zcd\  $F$-subalgebra $A\sub \M[n](K)$ can be put into
Wedderburn block form.
\end{cor}
\begin{proof} We put   $KA$ into Wedderburn block form
and then intersect down to~$A$. Explicitly, $A/J$ is \Zcd\ in the
semisimple part $S$ of $KA$, and $J$ is the intersection of $A$
with the part of $AK$ above the diagonal components.
\end{proof}

\subsection{Gluing Wedderburn blocks}

Let us investigate gluing in the Wedderburn block form. We start
with the semisimple part $\bigoplus A_u^{(\ell)}$ (of the blocks
along the diagonal). By definition, the only gluing occurs among
the $A_u^{(\ell)}$ for the same $u$.

\begin{rem} Suppose $\varphi_u^{\ell} \co A_u \to A_u^{(\ell)}$ is
the representation as above. Then for any $\ell, \ell'$ in $\{1,
\dots, t(u)\}$ we have the isomorphism $$\varphi_u^{\ell, \ell'} =
(\varphi_u^{\ell})^{-1}\varphi_u^{\ell'} \co A_u^{(\ell)} \to
A_u^{(\ell')}.$$\end{rem}

\begin{rem}\Label{glue1}
Let $1_u^{(\ell)}$ denote the unit element of $A_u^{(\ell)}$.
($1_u^{(\ell)}$ is then an idempotent of $\M[n](K)$, but we want
to emphasize its role in $A_u^{(\ell)}$.)   We want to understand
the isomorphisms $\varphi_u^{\ell, \ell'}$ in terms of their
action on the center of the block. First of all, by \Lref{centcl},
since $A$ is \Zcd, so is $\Cent{A}$, which contains $\sum _{u}
F_u \sum _{\ell} 1_u^{(\ell)}$. It follows (by Theorem
\ref{comZar}) that all identifications in the center come from
polynomial relations of Frobenius type, between pairs $e_\ell$ and
$e_\ell'$ (as $\ell, \ell'$ run between $1$ and $t(u))$; i.e., of
the form
$$ {\la_u^{(\ell)}}_{ii} -   ({\la_{u}^{({\ell'})}}_{ii})^q,$$
where $s$ is a power of $\card{F}$ (same $q$ for each $i =
1,\dots, n_u$). Here, and henceforth, ${\la_{u}^{(\ell)}}_{ij}$ is
the variable corresponding to the $(i,j)^{\mbox{th}}$ entry in the block
matrix $A_{u}^{(\ell)}$. This clearly is an instance of gluing.
\end{rem}

Since taking the power $q$ is not necessarily onto for $K$
infinite, Remark~\ref{glue1} is not symmetric, in the sense that
reversing direction from $\ell'$ to $\ell$ involves taking $q$
roots, which is possible in the variety but not over all of $K$.

In the relation above, if $\ell'=\ell$, i.e.,
${\la_u^{(\ell)}}_{ii} = ({\la_u^{(\ell)}}_{ii})^q$ holds, then we
can view $A_u^{(\ell)}$ as an algebra over a base field $F_u$ of
$q$ elements.

We continue from the center of diagonal blocks to the blocks
themselves.

\begin{defn}\Label{Fgng}
Suppose $\card{F} = q$. Two diagonal blocks $A_u^{(\ell)}$ and
$A_u^{(\ell')}$ of $A$ in $\M[n](K)$ have {\bf Frobenius gluing}
of {\bf exponent} $d$ if there are $n_u\times n_u$ matrix units of
$\M[n](K)$ such that, for large enough $\kappa$, writing
$e_{ij}^{(\ell)}$ for the corresponding $n_u \times n_u $ matrix
units in $A_u^{(\ell)}$, the isomorphism $\phi_u^{\ell,\ell'}$
identifies $\sum \a _{ij}^{q^{d+\kappa}} e_{ij}^{(\ell)}$ (in
$A_u^{(\ell)})$ with $\sum \a _{ij}^{q^{\kappa}} e_{ij}^{(\ell')}$
(in $A_u^{(\ell')})$.

This definition can be extended to gluing an arbitrary number $t$
of blocks.
\end{defn}

In this definition, since we may take $q$-roots, $\kappa$ can be
chosen to be $\max\set{0,-d}$. We could have $d = 0$, in which
case we call this {\bf identical gluing}. The same considerations
hold for an arbitrary number $t$ of glued blocks; the smallest
exponent may always be assumed to be $0$.
Soon we shall see that the only possible gluing on diagonal blocks
is Frobenius.

Let us consider the general situation. Since \Zcd\ algebras are
defined in terms of polynomial relations on the algebraically
closed field $K$ and all gluing is via a homomorphism from  $K$
to itself, the gluing must come from a homomorphism defined by a
polynomial.

\begin{rem}\Label{glu00}
Gluing is possible only between matrix blocks of the same size,
whose centers have the same cardinality.

When   $\cha (F)=0 $, the only polynomial homomorphism is the
identity, so every gluing is identical.
\end{rem}

\begin{prop}\Label{pass}
If $F$ is an infinite field, any variety of $F$-algebras contains
a \Zcd\ algebra whose gluing is identical.%
\end{prop}
\begin{proof} %
The \fcr\ $AK$ is in the same variety, by \Pref{finf} and
\Lref{samevar}.
\end{proof}

When $F$ is a finite field, we must also contend with the
Frobenius endomorphism, as illustrated in \Eref{unglued}, which we
note is preserved when we pass to the \Zcr.

The point of \Dref{Fgng} is that all corresponding entries in
these blocks are glued in exactly the same way (although not
necessarily by the identity map). This is one way in which the
theory is considerably richer in characteristic $p$  than in
characteristic~$0$.

\begin{thm}\Label{Frobglue}
\Label{glue} Suppose $A \subseteq \M[n](K)$ is a \Zcd\ algebra,
with
$$A/J = A_1 \times \cdots \times A_k,$$ a direct product of $\,k$
simple components. Then we can choose the matrix units of
$\M[n](K)$ in such a way that $A$ has Wedderburn block form, and
all identifications among the diagonal blocks are Frobenius gluing.
\end{thm}
\begin{proof}
Fix $u = 1,\dots, k$. Fixing a set of $n_u \times n_u$ matrix
units $\{ e_{ij} : 1 \le i,j \le n_u\}$ of $A_u^{(1)}$, we then
have the corresponding set of $n_u \times n_u$ matrix units $\{
\varphi_u^{1, \ell'}(e_{ij}) : 1 \le i,j \le n_u\}$ of
$A_u^{(\ell)}$. We do this for each $u$, and by \cite[Proposition
1.1.25]{Row2} all of these matrix units can be combined and
extended to a set of matrix units for $\M[n](K)$.

Now any matrix $\sum _{i,j =1}^{n_u} \a _{ij} e_{ij}^{(\ell)}$ of
$A_u^{(\ell)}$ is glued (via $\varphi_u^{\ell, \ell'}$) to $$\sum
_{i,j =1}^{n_u} \varphi_u^{\ell, \ell'}(\a _{ij}1_u)^{(\ell)}
e_{ij}^{(\ell')}\in A_u^{(\ell')}.$$
But, by Remark~\ref{glue1},
there is some $q = q (\ell, \ell')$ such that $$\varphi_u^{\ell,
\ell'}(\a \sum _{i =1}^{n_u}  e_{ii}^{(\ell)}) = \varphi_u^{\ell,
\ell'}(\a 1_u^{(\ell)}) = \a ^q 1_u^{({\ell'})} $$ %
(or visa versa, as noted above). Hence $\sum _{i,j =1}^{n_u} \a
_{ij} e_{ij}^{(\ell)}$ is glued to $\sum _{i,j =1}^{n_u} \a
_{ij}^q e_{ij}^{(\ell')}$, as desired.
\end{proof}

\subsection{Standard notation for Wedderburn blocks}\Label{ss:not}

We now change the point of view somewhat, and we write each diagonal
Wedderburn block as $B_1, \dots, B_m$ in the order in which they
appear on the diagonal. Thus $m = \sum_{u=1}^k t(u)$, where $t(u)$
is the number of blocks in the \th{u} glued component. Likewise
for $r<s$ we define the block $B_{rs} = B_r A B_s$. Any $B_{rs}$
can be viewed as a matrix block; in particular $B_{rr} = B_r$.
{}From this point of view, $B_{rs}$ is a $B_{r},B_{s}$-bimodule.
Each $B_r$ is a subalgebra of $\M[n](K)$, although in general it
is not contained in $A$. Letting $B = \sum_{r\leq s} B_{rs}$, we
have the following inclusions:
$$A \sub KA \sub B \sub \M[n](K).$$
\begin{exmpl}\Label{4.16}
When $F$ is finite, $A = \set{\smat{\alpha}{b}{0}{\alpha}
\suchthat \alpha \in F,\, b \in K}$ is in Wedderburn block form.
Then $$KA = \set{\smat{a}{b}{0}{a} \suchthat a, b \in K}$$
and $B
=\set{\smat{a}{b}{0}{a'} \suchthat a, a', b \in K}$, so the
inclusions $A \subset KA \subset B \subset \M[2](K)$ are all
strict.
\end{exmpl}

Let $T_1 \cup \cdots \cup T_k$ be the gluing partition of
$\set{1,\dots,m}$, namely $r \in T_u$ if, in the notation of
\Dref{wbf}, $B_r = A_{u}^{(\ell)}$ for some $\ell = 1,\dots,t(u)$.
Thus the \th{u} component of $A/J$ embeds as $\varphi_u \co A_u
\ra \bigoplus_{r \in T_u} B_{r}$. We let $\tau \co \set{1,\dots,m}
\ra \set{1,\dots,u}$ denote the quotient map, associating to every
index $r$ the gluing class $u$ of the block $B_r$; thus $r \in
T_{\tau(r)}$.

As always, $A$ is an algebra over a field $F$ of order $q$, a
prime power. (We also permit $F$ to be infinite, although this
case is easier.) We write $F_u$ for the field of scalar matrices
of $B_{rr}$, where $u = \tau(r)$. When finite, $\card{F_u} =
q^{d_u}$ for some number $d_r$.

\begin{rem}\Label{thisone}
Suppose $B_{rr}$ and $B_{ss}$ are glued blocks, with center of
order $q^{d_u}$ for $u = \tau(r) = \tau(s)$. If $B_{rr}$ and
$B_{ss}$ are glued via the Frobenius endomorphism $a \mapsto
a^{q^d},$ note that $d$ is only well defined modulo $d_u$.
\end{rem}

\subsection{Sub-Peirce decomposition and the \fcr}\Label{ss:sP}

Given a \Zcd\ algebra $A$ over $F$, represented in $\M[n](K)$, for
$K$ infinite, we have the primitive idempotents $\hat e_u$ of $A$
($u = 1,\dots,k$), which give rise to the Peirce decomposition $A
= \bigoplus \hat{e}_u A \hat{e}_v$.

Each idempotent decomposes as a sum $\hat e_u = \sum_{r \in T_u}
e_r$ of idempotents of $\M[n](K)$, where $T_u$ are defined in the
previous subsection, and we have the Peirce decomposition $B
=\bigoplus e_r A e_s$. This is a fine decomposition, and in general,
$e_r A e_s$ is not contained in $A$. Nevertheless, we do have the
following observation.

\begin{rem}\Label{forsepar}
Suppose $A = \prod _{i=1}^m F_i$ is a commutative semisimple
algebra and $a = (\a _i) \in A$ is written as $a = \sum a_j$, where
each $a_j$ is the sum of those Frobenius components of $a$ that
are glued. Then each $a_j \in \cl A$.
\end{rem}

\begin{defn}\Label{idemptype}
A primitive idempotent $\hat{e}$ of $A$ is of {\bf finite} (resp.\
{\bf infinite}) type if the base centers $F_u$ of the
corresponding glued blocks $e_r \M[n](K) e_r$ are finite (resp.\
infinite).
\end{defn}

Note that by \Pref{finf}, $F_u \cong K$ for any idempotent of
infinite type, although there may fail to be a natural action of
$K$ on the $F_u$ because of non-identity gluing.

\begin{exmpl}\Label{4.20}
\begin{enumerate}
\item The primitive idempotents of $$A =
\left\{\left(\begin{array}{ccc}
\alpha & 0 & x \\ %
0 & \beta & y \\ %
0 & 0 & \alpha^q
\end{array}\right): \alpha, \beta, x, y \in K\right\}$$
are $e_{11}+e_{33}$ and
$e_{22}$. Numerating the blocks into gluing components by setting
$T_1 = \set{1,3}$ and $T_2 = \set{2}$, we have that $F_1, F_2
\cong K$; however, scalar multiplication by $K$ does not preserve
$F_2$. Indeed, $A$ is not a $K$-algebra: $\dim(A) = 4$, while
$\dim(KA) = 5$.

\item Let $A = \left\{\left(\begin{array}{cccc}
\alpha & x & y & \lam x\\ %
0 & \beta & z & 0\\ %
0 & 0 & \alpha & 0 \\
0 & 0 & 0 & \beta
\end{array}\right): \alpha, \beta, x,y,z \in K\right\}$, where
$\lam \in K$ is fixed. The glued blocks are $T_1 = \set{1,3}$ and
$T_2 = \set{2,4}$. Accordingly, $A = A_{11}\oplus A_{12}\oplus
A_{21}\oplus A_{22}$, where $A_{11} = K(e_{11}+e_{33})+Ke_{13}$,
$A_{12} = K(e_{12}+\lam e_{14})$, $A_{21} = K e_{23}$ and $A_{22}
= K(e_{22}+e_{44}) + K e_{24}$.
\end{enumerate}
\end{exmpl}

We would like to refine this description by comparing the
Wedderburn decompositions of $A$ and its \fcr\ $KA$ (which may
have more primitive idempotents), even though $A$ and $KA$ need
not be PI-equivalent.

Break every gluing class $T_u$ ($u = 1,\dots,k$) into a disjoint
union $T_u = T_u^{(1)} \cup \dots \cup T_u^{(c_u)}$, where blocks
$B_r, B_s$ are in the same component $T_u^{(\mu)}$ if and only if
they are glued by an identical gluing. For example, in
\Eref{4.20}(1) the decomposition is $T_1 = \set{1} \cup \set{3}$.
The idempotents $\hat{e}_u$ decompose, accordingly, as
\begin{equation}\Label{hatbar}
\hat{e}_u = \sum_{\mu = 1}^{c_u} \bar{e}_u^{(\mu)},
\end{equation}
where $\bar{e}_u^{(\mu)} = \sum_{r \in T_u^{(\mu)}} e_r$. Although
$\bar{e}_u^{(\mu)}$ are not in $A$, these elements do belong to
$KA$ (since $K$ is infinite, allowing for a Vandermonde argument).
Therefore, we define:

\begin{defn}
The {\bf sub-Peirce decomposition} of $A$ is the restriction to
$A$ of the Peirce decomposition of $KA$; \cf~equation~(\ref{PDM}).
Namely,
$$A \subseteq \bigoplus A_{uv}^{(\mu\mu')}, \qquad A_{uv}^{(\mu \mu')}
    = \bar{e}_u^{(\mu)} A \bar{e}_{v}^{(\mu')},$$
where the sum ranges over $u,v = 1,\dots,k$, $\mu = 1,\dots,c_u$
and $\mu' = 1,\dots,c_v$. We stress once more that this is not a
decomposition of $A$, as the $A_{uv}^{(\mu\mu')}$ are contained in
$KA$, but not in $A$ in general.
\end{defn}

Notice that if all the gluing in $A$ are via the identity map, in
particular (by \Pref{pass}), if $A$ is a $K$-algebra, then the
sub-Peirce decomposition is identical to the Peirce decomposition.

\begin{exmpl}\Label{4.21}
Let $A = \left\{\left(\begin{array}{cccc}
\alpha & x & y & z\\ %
0 & \alpha^q & x' & y'\\ %
0 & 0 & \alpha & x'' \\
0 & 0 & 0 & \alpha^{q}
\end{array}\right): \alpha, x,x',x'',y,y',z \in K\right\}$.
There is one glued component, namely $T_1 = \set{1,2,3,4}$, which
decomposes with respect to identical gluing as $T_1 =
\set{1,3}\cup \set{2,4}$. The corresponding idempotent
decomposition is $\hat{e}_1 = \bar{e}_1^{(1)} + \bar{e}_1^{(2)}$,
where $\hat{e}_1 = 1$, $\bar{e}_1^{(1)} = e_{11}+e_{33}$ and
$\bar{e}_1^{(2)} = e_{22}+e_{44}$. The sub-Peirce components are
$A_{11}^{(11)} = K\bar{e}_1^{(1)} + K e_{13}$, $A_{11}^{(12)} = K
e_{12}+Ke_{14}+Ke_{34}$, $A_{11}^{(21)} = Ke_{23}$ and
$A_{11}^{(22)} = K\bar{e}_1^{(2)} + Ke_{24}$ (similarly to the
Peirce components in \Eref{4.20}(2)).
\end{exmpl}

{}From one point of view, the \fcr\ erases all the subtlety
introduced by the finiteness of $F$, as we see in the next
observation.

\begin{rem}\Label{gluetype} Identity gluing in $A$ is preserved in $KA$;
however, it may
happen that $A = \cl{A}$ and $A$ has only identical gluing, while
$A \subset KA$ (see \Eref{4.16}).

On the other hand, non-identity (Frobenius) gluing  for $A$ is
unglued in $KA$, as seen by applying a Vandermonde argument since $K$
is infinite. Thus $KA$ only has identical gluing.

Viewed in terms of the Peirce decomposition, a Peirce component
$A_{uu}$ of $A$ may ramify in $KA$, and the corresponding
primitive idempotent in $A$ becomes a sum of orthogonal
idempotents in $KA$. Thus, the sub-Peirce decomposition of $A$
consists of identity-glued components of the Peirce decomposition
of $KA$.

{}From the point of view of PI's, $KA$ satisfies all {\it
multilinear} identities of $A$, although it may lose identities
arising from Frobenius automorphisms, such as $x^2y-yx$ in the
example  $A = \set{\smat{a}{b}{0}{a^2} \suchthat a, b \in \F_4}$.
\end{rem}

It turns out that non-identical gluing permits us to refine the
decomposition further, and this is our next goal.

\subsection{Relative exponents}\Label{ss:rd}

\begin{defn}\Label{reldeg}
Let $B_r$ and $B_{r'}$ be two glued blocks (whose centers thus
have the same cardinality). By \Tref{Frobglue}, we may assume the
blocks are glued by Frobenius gluing of some exponent (\cf\
\Dref{Fgng}), which we denote as $\exp(B_{rr'})$ and call the
{\emph{relative Frobenius exponent}} of $B_{rr'}$. This is
understood to be zero if $F$ is infinite. In fact, $\exp(B_{rr'})$
is only well defined modulo the dimension of $F_r = F_{r'}$ over
$F$ (where we interpret `modulo infinity' as a mere integer).
\end{defn}

The relative Frobenius exponents are used to define equivalence
relations on vectors of glued indices, as follows.
\begin{defn}\Label{eqdef}
Recall the definition of $T_u$ from Subsection~\ref{ss:not}. For every
$1\leq u,v \leq k$, we let $T_{u,v} = \set{(r,s) \in T_u \times
T_v \suchthat r\leq s}$, and define an equivalent relation on
$T_{u,v}$ by setting $(r,s) \sim (r',s')$ iff $\exp(B_{rr'})
\equiv \exp(B_{ss'})$ modulo $\gcd(d_u,d_v)$. (Recall that $d_u$
is the dimension of the center of the \th{u} component over $F$.)

More generally, for every $t$-tuple $1 \leq u_1,\dots,u_t \leq k$,
we set $$T_{u_1,\dots,u_t} = \set{(r_1,\dots,r_t) \in
T_{u_1}\times \cdots \times T_{u_t} \suchthat r_1 \leq \cdots \leq
r_t},$$ and define an equivalence relation on $T_{u_1,\dots,u_t}$
by $(r_1,\dots,r_t) \sim (r_1',\dots,r_t')$ if the values
$\exp(B_{r_1^{\,} r_1'}), \dots, \exp(B_{r_t^{\,} r_t'})$ are all
equivalent modulo $\gcd(d_{u_1},\dots,d_{u_t})$. We call $t$ the
\defin{length} of $\gamma$.
\end{defn}

\begin{rem}\Label{trans}
\begin{enumerate}
\item Relative exponents can be computed with respect to a fixed
block in the gluing component, \ie\ $\exp(B_{rr'}) =
\exp(B_{r_0r'}) -\exp(B_{r_0r})$. In particular, $\exp(B_{rr'})=
-\exp(B_{r'r})$, and $\exp(B_{rr''}) =
\exp(B_{rr'})+\exp(B_{r'r''})$. \item In a `diagonal' set
$T_{u,u}$, $\exp(B_{rr'}) \equiv \exp(B_{ss'})$ iff $\exp(B_{rs})
\equiv \exp(B_{r's'})$, and so one can read the equivalence
relation directly from the matrix of relative exponents. \item
Likewise for any $t$, if $u_1 = \cdots = u_t = u$, then the
equivalence relation on $T_{u,\dots,u}$ is given as follows:
$(r_1,\dots,r_t) \sim (r_1',\dots,r_t')$ iff $\exp(B_{r_i
r_{i+1}}) \equiv \exp(B_{r_i'r_{i+1}'})$ modulo $d_u$ for $i =
1,\dots,t-1$. \item If $B_{r}$ and $B_{\bar{r}}$ are identically
glued, then obviously $\exp(B_{rs}) = \exp(B_{\bar{r}s})$ for any
$s$. In particular, $(r_1,\dots,r_{i-1},r,r_{i+1},\dots,r_t) \sim
(r_1,\dots,r_{i-1},\bar{r},r_{i+1},\dots,r_t)$ whenever $r_{i-1}
\leq r,\bar{r} \leq r_{i+1}$.
\end{enumerate}
\end{rem}

A word of caution: When all the Peirce idempotents in a given
sub-Peirce vector $(u_1, \dots, u_t)$ correspond to fields
$F_{u_i}$ of the same size (such as all having infinite type),
then the equivalence class of this vector is uniquely determined
by the equivalence classes of the pairs $(u_i, u_{i+1})$. However,
this can fail when the $F_{u_i}$ have differing sizes (such as
some finite and some infinite), because of the ambiguity arising
from the differing exponents of the Frobenius automorphisms.

If $(r_i,r_{i+1}) \sim (r_i',r_{i+1}')$ for each $i =
1,\dots,t-1$, then the respective relative exponents are
equivalent modulo $\gcd (d_i,d_{i+1})$, so in particular they are
all equivalent modulo $\gcd(d_1,\dots,d_t)$, and so
$(r_1,\dots,r_{t}) \sim (r_1',\dots,r_{t}')$.

On the other hand, equivalence modulo $\gcd(d_1,\dots,d_t)$ does
not in general imply any equivalence modulo $\gcd(d_i,d_{i+1})$,
so, for example, $(r_1,r_2,r_3) \sim (r_1',r_2',r_3')$ does not
even imply $(r_1,r_2) \sim (r_1',r_2')$.

\begin{defn}\Label{defcomp}
We define the \defin{composition} of equivalence classes, in the
spirit of matrix units, as follows. Suppose $\gamma \sub
T_{u_1,\dots,u_t}$ and $\gamma' \sub T_{v_1,\dots,v_{t'}}$. If
$u_t \neq v_1$, let $\gamma
* \gamma' = \emptyset$; and if $u_t = v_1$, let $$\gamma * \gamma'
= \set{(r_1,\dots,r_{t-1},r_t,s_2,\dots,s_{t'}) \suchthat
(r_1,\dots,r_{t-1},r_t) \in \gamma, \, (r_t,s_2,\dots,s_{t'}) \in
\gamma'}.$$ The composition $\gamma \circ \gamma'$ is defined as
the set of equivalence classes (of length $t+t'-1$) of
$T_{u_1,\dots,u_t,v_2,\dots,v_{t'}}$ which are contained in
$\gamma * \gamma'$.
\end{defn}

\subsection{The relative Frobenius decomposition}

Let $a \in A$ be an element in a Peirce component $\hat{e}_u A
\hat{e}_v$ of $A$. Applying the matrix block component
decomposition $a = \sum_{(r,s) \in T_{u,v}} a_{rs}$ (with $a_{rs}
\in B_{rs}$), we have $a = \sum a^{\gamma}$, where %
\begin{equation}\Label{agamma}
a^{\gamma} = \sum_{(r,s) \in \gamma} a_{rs} \end{equation} and
$\gamma$ ranges over the equivalence classes of $T_{u,v}$ defined
in \Dref{eqdef}. Thus, letting $A^{\gamma} = \set{a^{\gamma}
\suchthat a \in \hat{e}_u A \hat{e}_v}$, we have the
\defin{relative Frobenius decomposition}
\begin{equation}\Label{relFd}
A = \sum_{1\leq u,v \leq k} \,\sum_{\gamma \sub T_{u,v}} A^{\gamma}.
\end{equation}

\begin{prop}\Label{QQQ}
If $a \in \cl A$, then $a^\gamma \in \cl A$ for each equivalence
class $\gamma$ in every $T_{u,v}$.
\end{prop}
\begin{proof}
We need to show that each glued component $a^\gamma$ is in $\cl
A$. Let $M_{rs}$ be the $F$-subspace of $\M[n](K)$ spanned by the
$a_{rs}e_{rs}$. The natural vector space isomorphism $F \to
M_{rs}$ can be used to transfer the algebra structure of $F$ to
$M_{rs}$; i.e., $$(\a a_{rs}e_{rs})(\beta a_{rs}e_{rs}) = \a \beta
a_{rs}e_{rs}.$$ Note that each $M_{rs}$ is closed under these
operations. With respect to this structure, $M = \bigoplus
_{r,s}M_{rs}$ becomes a commutative $F$-algebra (defining the
operations componentwise), and its subalgebra $\tilde M$
corresponding to $a^\gamma$ is semiprime \Zcd, and thus has only
Frobenius gluing, in view of Remark \ref{glu00}. But this implies
$\tilde M \subseteq \cl A$, as desired.
\end{proof}

We see that the relative Frobenius decomposition is finer than the
Peirce decomposition, but coarser than the sub-Pierce
decomposition (which, strictly speaking, is not a decomposition of
$A$, since the components only belong to $KA$).

\begin{cor}\Label{offglue}
Two blocks of the \fcr\ $KA$ are in the same component in the
sub-Peirce decomposition iff they are identically glued. Thus, in
the relations defining the algebra $A$, any two off-diagonal
blocks $B_{rs}$ and $B_{r's'}$ with $(r,s) \not \sim (r',s')$ can
be separated (see \Rref{directsum} below).
\end{cor}

\begin{rem}\Label{moved}
Suppose $f$ is a relation of weak Frobenius type on $A$ whose
variables involve the blocks $B_{r_1, s_1}, \dots, B_{r_\nu,
s_\nu}$. Assume $f$ cannot be concluded from relations on the same
blocks, with fewer variables. Then the following facts hold:
\begin{enumerate}
\item The diagonal blocks $B_{r_1}, \dots , B_{r_\nu}$ are glued.
(Indeed, suppose some $a \in A$ satisfies the weak Frobenius
relation
$$\sum _{i=1}^n \sum _{j\ge 1} c_{ij} (a^{i,j}_{r,s})^{q_{r,s}} =
0.$$ %
If some $B_r$ and $B_{r'}$ are not glued, then we have a diagonal
element $d\in A$ with the identity $1$ in the $B_r$ block and $0$
in the $B_{r'}$ block. Then $da=0$, contrary to the minimality of
the quasi-linear relation defining the gluing.) Thus, applying
some graph theory cuts down the number of generating polynomial
relations even further.

\item Each diagonal block $B_r$ is defined over a field whose
order is at most the maximal $q_{r,s}$. (Multiply diagonal
elements of $A$ by the elements in the given glued components of
$A$, and apply a Vandermonde argument.)
\end{enumerate}
\end{rem}

\begin{exmpl}\Label{4.25}
Let $A$ be the algebra of \Eref{4.21}. For $u = v = 1$, the
relative exponents of $B_{rs}$ is the \th{(r,s)} entry in the
antisymmetric matrix $\left(\begin{array}{cccc}
0 & 1 & 0 & 1\\ %
\cdot & 0 & -1 & 0\\ %
\cdot & \cdot & 0 & 1 \\
\cdot & \cdot & \cdot & 0
\end{array}\right)$; see \Rref{trans}(2). The equivalence
relation on $T_{1,1}$ has classes $\set{(1,2),(1,4),(3,4)}$,
$\set{(2,3)}$ and $\set{(1,1),(1,3),(2,2),(2,4),(3,3),(4,4)}$. In
the notation of \Eref{4.21}, the relative Frobenius decomposition
is
$$A = ( A_{11}^{(11)} + A_{11}^{(22)} ) \oplus A_{11}^{(12)}
\oplus A_{11}^{(21)}.$$
\end{exmpl}

\begin{exmpl}
Now take $A = \left\{\left(\begin{array}{cccc} %
\alpha & x & y & z\\ %
0 & \beta^q & x' & y'\\ %
0 & 0 & \beta & x'' \\
0 & 0 & 0 & \alpha^{q}
\end{array}\right): \alpha, \beta, x,x',x'',y,y',z \in K\right\}$.
There are two glued components, $T_1 = \set{1,4}$ and $T_2
=\set{2,3}$. The non-diagonal relative Frobenius exponents are
$\exp(B_{14}) = 1$ and $\exp(B_{23}) = -1$, as in \Eref{4.25}. The
equivalence components are $T_{1,1} = \set{(1,1),(4,4)} \cup
\set{(1,4)}$, $T_{1,2} = \set{(1,2)}\cup \set{(1,3)}$, $T_{2,1} =
\set{(2,4)}\cup \set{(3,4)}$ and $T_{2,2} = \set{(2,2),(3,3)} \cup
\set{(2,3)}$. For this algebra, the relative Frobenius
decomposition recaptures the full sub-Peirce decomposition.
\end{exmpl}

\subsection{Higher length decomposition}
The same proof as in \Pref{QQQ} yields the following more
intricate result:
\begin{rem}\Label{abitmore}
Let $t \geq 2$ and $1\leq u_1,\dots,u_t \leq k$. Let $a =
\hat{e}_{u_1} a_1 \hat{e}_{u_2} \cdots \hat{e}_{u_{t-1}} a_{t-1}
\hat{e}_{u_t} \in \hat{e}_{u_1} A \hat{e}_{u_2} \cdots
\hat{e}_{u_{t-1}} A \hat{e}_{u_t}$, where $a_1,\dots,a_{t-1} \in
A$.

For every equivalence class $\gamma \sub T_{u_1,\dots,u_t}$ (as
defined in \Dref{eqdef}) and $a$ as above, let $a^{\gamma} =
\sum_{(r_1,\dots,r_t) \in \gamma} {e_{r_1} a_1 e_{r_2} \cdots
e_{r_{t-1}} a_{t-1} e_{r_t}}$. Then $a^{\gamma} \in A$ and
(clearly) $a = \sum_{\gamma \sub T_{u_1,\dots,u_t}} a^{\gamma}$.
Writing $a_i = \sum_{rs} \alpha^{(i)}_{rs}e_{rs}$, we have that
\begin{equation}\Label{form}
a^{\gamma} = \sum_{(r_1,\dots,r_t) \in \gamma}
\alpha_{r_1r_2}^{(1)} \cdots \alpha_{r_{t-1}r_t}^{(t-1)}
e_{r_1r_t}.\end{equation}

Letting $A^{\gamma}  = \set{a^{\gamma} \suchthat a \in
\hat{e}_{u_1} A \hat{e}_{u_2} \cdots \hat{e}_{u_{t-1}} A
\hat{e}_{u_t}}$ (where we implicitly take advantage of the fact
that $\gamma$ determines $(u_1,\dots,u_t)$), we have the $t$-fold
relative Frobenius decomposition
\begin{equation}\Label{compg} A = \sum_{1\leq u_1,\dots,u_t \leq
k} \,\sum_{\gamma \sub T_{u_1,\dots,u_t}} A^{\gamma}.%
\end{equation}
\end{rem}

While \Eq{relFd} is clearly a direct sum, this is no longer the
case for $t \geq 2$ in \eq{compg}, as demonstrated in \Eref{six}
below.

\begin{rem}\Label{nest}
The decomposition of \Rref{abitmore} becomes finer as the length
of the classes increases. Indeed, for $\gamma$ of length $t+1$ and
$1 < i < t+1$, let $\pi_i$ denote the map $\set{1,\dots,m}^{t+1}
\ra \set{1,\dots,m}^{t}$, forgetting the \th{i} entry. If $\gamma
\sub T_{u_1,\dots,u_{t+1}}$, then $\pi_i(\gamma)$ is contained in
an equivalence class of $T_{u_1,\dots,u_{i-1},u_{i+1},\dots,
u_{t+1}}$. For any equivalence class $\hat{\gamma} \sub
T_{u_1,\dots,u_t}$, one uses $A = \sum A \hat{e}_{v} A$ to show
that
%
$$A^{\hat{\gamma}} = \sum_{v=1}^{k} \sum_{\gamma} A^{\gamma},$$
where, for each $v$, $\gamma$ ranges over all equivalence classes
of $T_{u_1,\dots,u_{i-1},v,u_i,\dots,u_t}$ such that
$\pi_i(\gamma) = \hat{\gamma}$ ($k$ is the number of gluing
components).
%
\end{rem}

Products of the components obtained in this manner can be easily
computed via the following formula.
\begin{rem}\Label{mulc}
If $\gamma$ and $\gamma'$ are equivalence classes of arbitrary
length, then we have the `thick filtration' formula $$A^\gamma \cdot
A^{\gamma'} = \sum_{\gamma'' \in \gamma \circ \gamma'}
A^{\gamma''}$$ (where the composition is defined in
\Dref{defcomp}). In particular, the decomposition of length $t+1$
is a refinement of the product of the basic decomposition, of
length $2$, with the decomposition of length $t$.
\end{rem}

\begin{rem}
If $A \sub \M[n](K)$ is an $F$-algebra (not necessarily \Zcd), we
can define $\check{A}$ to be the extension of $A$ in $KA$
generated by all sub-Pierce components of elements of $A$. Thus,
$A \sub \cl{A} \sub \check{A} \sub KA$.

For example, if $A = \set{ \smat{\alpha^{p^n}}{a}{0}{\alpha}:
\alpha \in F_1, \, a\in K}$ (cf.~ Example \ref{basexa2}(1)), then
$\check{A} =    \smat{F_2}{K}{0}{F_1},$ where $F_2 = \{ \a^p: \a
\in K_1\}.$

In particular, if $A\sub \M[n](K)$ is a generic algebra of a given
variety, generated by generic elements (see \Sref{sec:6} below),
then $\check{A}$ is naturally graded by the matrix components of
$\M[n](K)$, although $A$ itself is not graded.
\end{rem}

\subsection{Interaction with the radical}\label{ss:4.6}

The decompositions in (1)--(3) above can be refined further via
the decomposition $A = S \oplus J$ to the semisimple and radical
parts. For example $A_{(uv)} \sub J$ for $u \neq v$, but
$A_{(uu)}$ is a (nonunital) subalgebra of $A$, with
$\Rad(A_{(uu)}) = J \cap A_{(uu)}$.

\begin{rem}\Label{gluetype1}
Let $\gamma \sub T_{u_1,\dots,u_t}$ be an equivalence class, and
fix $1 \leq i < t$. By definition, $A^{\gamma} =
\set{a^{\gamma}}$, ranging over %
\begin{eqnarray*}
a & = & \hat{e}_{u_1} a_1 \hat{e}_{u_2} \cdots \hat{e}_{u_{t-1}}
a_{t-1} \hat{e}_{u_t} \\ & \in & \hat{e}_{u_1} A \cdots
\hat{e}_{u_i} A
\hat{e}_{u_{i+1}} \cdots A \hat{e}_{u_t} \\
& = & \hat{e}_{u_1} A \cdots \hat{e}_{u_i} S \hat{e}_{u_{i+1}}
\cdots A \hat{e}_{u_t} + \hat{e}_{u_1} A \cdots \hat{e}_{u_i} J
\hat{e}_{u_{i+1}} \cdots A
\hat{e}_{u_t}.\end{eqnarray*} %
If $u_i \neq u_{i+1}$, then
$\hat{e}_{u_i} S \hat{e}_{u_{i+1}} = 0$, so we may assume $a_i \in
J$. Otherwise, the computation shows that
$$a^{\gamma} = \sum_{(r_1,\dots,r_t) \in \gamma,\ r_i < r_{i+1}}
{e_{r_1} a_1 e_{r_2} \cdots e_{r_{t-1}} a_{t-1} e_{r_t}}$$ when
$a_i \in J$, while
$$a^{\gamma} = \sum_{(r_1,\dots,r_t) \in \gamma,\ r_i = r_{i+1}}
{e_{r_1} a_1 e_{r_2} \cdots e_{r_{t-1}} a_{t-1} e_{r_t}}$$ when
$a_i \in S$. In other words, we refine the
decomposition~\eq{compg} by separating the conditions $r_i \leq
r_{i+1}$ on an equivalence class $\gamma = \set{(r_1,\dots,r_t)}$
to one of the conditions $r_i = r_{i+1}$ or $r_{i} < r_{i+1}$.

We say the index $i$ has \defin{type} $0$ in $\gamma$ if $r_i =
r_{i+1}$ for every $(r_1,\dots,r_t) \in \gamma$; and has type $1$ if
$r_i < r_{i+1}$ for every $(r_1,\dots,r_t) \in \gamma$. For
example, if $u_i \neq u_{i+1}$, then $i$ has type $1$. An
equivalence class can be decomposed as a union $\gamma =
\gamma^{(0)} \cup \gamma^{(1)}$, where $\gamma^{(0)} =
\set{(r_1,\dots,r_t) \in \gamma \suchthat r_{i} = r_{i+1}}$ and
$\gamma^{(1)} = \set{(r_1,\dots,r_t) \in \gamma \suchthat r_i <
r_{i+1}}$. This process can be repeated for every $i$, and the
resulting sub-classes are called \defin{fully refined} classes.

One can multiply the components corresponding to
refined classes, as described for standard components in \Rref{mulc}. %
\end{rem}

\def\ww{{\omega}}

If $\gamma^*$ is a fully refined equivalence class, let the
\defin{weight} $\ww(\gamma^*)$ denote the number of indices $i$ of
type $1$ in $\gamma^*$.  In particular, any components having
$\ww(\gamma^*) $ greater than the nilpotence index of $J$ must be
zero. By construction, $A^{\gamma^*} \sub J^{\ww(\gamma^*)}$.
Moreover, $J^\ell = \sum_{\len(\gamma^*) = t,\ \ww(\gamma^*) \geq
\ell} A^{\gamma^*}$ for every $t \geq \ell$.

\begin{lem}
Suppose $\gamma^* \sub T_{u_1,\dots,u_k}$ is a refined equivalence
class, with the index $i$ having type $0$. Let $\gamma' \sub
T_{u_1,\dots,u_{i},u_{i+2},\dots,u_k}$ be the equivalence class
obtained by removing the $(i+1)$th entry from each vector in
$\gamma^*$. Then $A^{\gamma^*} = A^{\gamma'}$.
\end{lem}
\begin{proof} It is clear that $A^{\gamma^*} \sub A^{\gamma'}$, and if
$$a^{\gamma'} =
\sum_{(r_1,\dots,r_{i},r_{i+2},\dots,r_t) \in {\gamma'}} {e_{r_1}
a_1 e_{r_2} \cdots e_{r_{i}} a_{i+1} e_{r_{i+2}} \cdots
e_{r_{t-1}} a_{t-1} e_{r_t}}$$ for some
$a_1,\dots,a_{i-1},a_{i+1},\dots,a_{t-1} \in A$, then, taking
$a_{i} = 1$,
$$a^{\gamma^*} = \sum_{(r_1,\dots,r_{i},r_{i+1},r_{i+2},\dots,r_t)
\in {\gamma^*}} {e_{r_1} a_1 e_{r_2} \cdots e_{r_{i}} a_i
e_{r_{i+1}} a_{i+1} e_{r_{i+2}} \cdots e_{r_{t-1}} a_{t-1}
e_{r_t}}$$ is equal to $a^{\gamma'}$, since $e_{r_i} a_i
e_{r_{i+1}} = e_{r_i}$ for all vectors in $\gamma^*$.
\end{proof}

\begin{prop}\label{boundonlen}
Every homogeneous component of a fully refined equivalence class
has the form $A^{\gamma^*}$, where all indices in $\gamma^*$ have
type $1$. In particular $\len(\gamma^*) = \omega(\gamma^*) \leq
\nu$.
\end{prop}

\begin{exmpl}\Label{six}
Consider the $F$-algebra $$A = \set{
\left(\begin{array}{cccccc} %
a & 0 & *   & * & *        & * \\
0 & a & x   & y & *        & * \\
0 & 0 & a^q & x' & *        & * \\
0 & 0 & 0   & a & \alpha x & \alpha \alpha' y \\
0 & 0 & 0   & 0 & a^q      & \alpha' x' \\
0 & 0 & 0   & 0 & 0        & a \\
\end{array}\right) : a,x,x',y, *,\dots,* \in K
},$$ where $\alpha,\alpha' \in K$ are fixed, and $q = \card{F}$.
As in previous examples, there is one glued component $T_1 =
\set{1,2,3,4,5,6}$, which decomposes with respect to identical
gluing as $T_1 = \set{1,2,4,5} \cup \set{3,6}$. In the equivalence
relation of \Dref{eqdef}, $T_{11}$ decomposes into the three
equivalence classes: $\gamma_1 =
\set{(1,3),(1,5),(2,3),(2,5),(4,5)}$ (with relative exponent $1$),
$\gamma_{-1} = \set{(3,4),(3,6),(5,6)}$ (with relative exponent
$-1$), and the complement $\gamma_0$, whose elements are the $13$
pairs of relative exponent $0$. (Throughout this example, we use
relative exponents as indices.) The relative Frobenius
decomposition of a general element, described in \Eq{relFd}, is
$$\left(\begin{array}{cccccc} %
a & 0 & 0   & * & 0        & * \\
0 & a & 0   & y & 0        & * \\
0 & 0 & a^q & 0 & *        & 0 \\
0 & 0 & 0   & a & 0 & \alpha \alpha' y \\
0 & 0 & 0   & 0 & a^q      & 0 \\
0 & 0 & 0   & 0 & 0        & a \\
\end{array}\right)
+
\left(\begin{array}{cccccc} %
0 & 0 & *   & 0 & *        & 0 \\
0 & 0 & x   & 0 & *        & 0 \\
0 & 0 & 0   & 0 & 0       & 0\\
0 & 0 & 0   & 0 & \alpha x & 0 \\
0 & 0 & 0   & 0 & 0        & 0 \\
0 & 0 & 0   & 0 & 0        & 0 \\
\end{array}\right)
+
\left(\begin{array}{cccccc} %
0 & 0 & 0   & 0 & 0        & 0 \\
0 & 0 & 0   & 0 & 0        & 0 \\
0 & 0 & 0   & x' & 0       & * \\
0 & 0 & 0   & 0 & 0 & 0 \\
0 & 0 & 0   & 0 & 0        & \alpha' x' \\
0 & 0 & 0   & 0 & 0        & 0 \\
\end{array}\right),
$$
and we denote the respective summands in the decomposition of $A$
as $\Gamma_0 + \Gamma_1 + \Gamma_{-1}$.

Next, we describe the decomposition corresponding to $t = 3$ in
\Rref{abitmore}. The set $T_{1,1,1}$, consisting of $\binom{8}{3}
= 56$ triples, decomposes into the $7$ equivalence classes:
$\gamma_{0,0} = [(111)]$, 
$\gamma_{0,1} = [(113)]$, 
$\gamma_{1,0} = [(133)]$, 
$\gamma_{1,-1} = [(134)]$, 
$\gamma_{0,-1} = [(334)]$, 
$\gamma_{-1,0} = [(344)]$ 
and $\gamma_{-1,1} = [(345)]$. 
Let $\Gamma_{\delta\delta'}$ denote the component in \eq{compg}
corresponding to $\gamma_{\delta\delta'}$. Computing via formula
\eq{form} we find that $\Gamma_{0,1} = \Gamma_{1,0} = \Gamma_1$,
that $\Gamma_{0,-1} = \Gamma_{-1,0} = \Gamma_{-1}$; and that
$\Gamma_{0,0} = \Gamma_0$. On the other hand $\Gamma_{-1,1} =
Ke_{35}$ and $\Gamma_{1,-1} = K e_{14}+ K e_{16} + Ke_{26} +
K(e_{24}+\alpha \alpha' e_{46})$ are proper subspaces of previous
components; however, we do not get a finer decomposition of $A$.

Applying the decomposition $A = S+J$, as in \Rref{gluetype1}, we
observe the following. The class $\gamma_{00} = [(111)]$ breaks
down to four sub-classes, namely $\gamma_{00} = (\gamma_{00}^{==})
\cup (\gamma_{00}^{=<}) \cup (\gamma_{00}^{<=}) \cup
(\gamma_{00}^{<<})$, with the obvious interpretation. For example,
$\gamma_{00}^{<<} = \set{(146),(246)}$. The corresponding
components are $\Gamma_{00}^{==} = S$, the semisimple subalgebra;
$\Gamma_{00}^{<<} = K e_{16}+K e_{26}$; and $\Gamma_{00}^{=<} =
\Gamma_{00}^{<=} = J \cap \Gamma_{00}$. Similar decomposition can
be applied to the other classes, although in every case some
sub-classes are empty. For example, $\gamma_{0,-1} =
(\gamma_{0,-1}^{=<}) \cup (\gamma_{0,-1}^{<<})$, with
$\Gamma_{0,-1}^{<<} = K e_{36}$ and $\Gamma_{0,-1}^{=<} =
\Gamma_{-1}$.
\end{exmpl}

\subsection{Summary}

Let $A \subseteq \M[n](K)$ be a \Zcd\ $F$-subalgebra, written in
Wedderburn block form (\Cref{block2}). In this section we have
discussed four useful decompositions. We follow the notation of
\Dref{wbf} and the idempotents defined in Subsection~\ref{ss:sP}.

\begin{enumerate}
\item The Peirce decomposition of $A$ is given by $A = \bigoplus
{\hat{e}}_u A \hat{e}_v$ for $1\leq u,v \leq k$. Thus every
element in $A$ can be written as $a = \sum_{u,v} a_{(uv)}$, with
$a_{(uv)} \in A$.
\item The relative Frobenius decomposition is a finer decomposition of
$A$: Every $a_{(uv)} \in \hat{e}_u A \hat{e}_v$ decomposes as
$a_{(uv)} = \sum_{\gamma} a^{\gamma}$, where $a^{\gamma}$ is
defined in \eq{agamma} and $\gamma \sub T_{u,v}$ ranges over the
equivalence classes of \Dref{eqdef}. The components $a^{\gamma}
\in A$ by \Pref{QQQ}. This can be refined further by
\Rref{abitmore}.
\item Each idempotent $\hat{e}_u$ is a sum of idempotents
$\sum_{\mu = 1}^{c_u} \bar{e}_u^{(\mu)}$, as in \Eq{hatbar}, where
each idempotent $\bar{e}_u^{(\mu)}$ corresponds to the
$\mu$-components of identical gluing. The Peirce decomposition of
$KA$ is $KA = \sum K A_{uv}^{(\mu\mu')}$, which gives the
decomposition $A = \bigoplus A_{uv}^{(\mu \mu')}$, with
$A_{uv}^{(\mu\mu')} \sub KA$.
\item[] If $F$ is infinite, then the decompositions (1), (2) and
(3) are identical.
\item Finally, we have the decomposition $A = \bigoplus A_{rs}$ of matrix
blocks, where $A_{rs} = e_r A e_s \sub B$, and $e_r$ are the block
idempotents defined by the Wedderburn block form. Every
$\bar{e}_u^{(\mu)}$ breaks down as a sum of such idempotents, as
demonstrated in \Eref{4.21}.
\item The decomposition in (4) can be refined by taking
equivalence classes of any length $\geq 2$, as described in
\Rref{abitmore}.
\item Finally, one may apply the Wedderburn decomposition, or
equivalently replace the equivalence classes by fully refined ones
of length $\leq \nu$; \cf\ Subsection~\ref{ss:4.6}.
\end{enumerate}

\section {Explicit generators for polynomial relations}\Label{s:explicit}

We are ready for a fairly precise description of the polynomial
relations of a \Zcd\ algebra $A$, with radical $J$.

\begin{rem}\Label{directsum}
Suppose $V$ is an $F$-subspace of a $K$-algebra $B$. If $V = V_1
\oplus V_2$ is a direct sum, then any polynomial relation of $V$
is a sum of polynomial relations of $V_1$ and $V_2$. \end{rem}

In view of this observation, taking the ``Wedderburn
decomposition'' $A = J \oplus S$ inside the Wedderburn
decomposition of $B$, we see that the polynomial relations of $A$
are generated by the polynomial relations of $J$ and the
polynomial relations of $S$.

The polynomial relations of $S$ come from gluing, which involves
$\Cent{S}$, a commutative algebra.

This leaves the polynomial relations of $J$, whose identifications
may be considerably more  intricate, especially in the presence of
Frobenius gluing. In view of Theorem~ \ref{linear}, off-diagonal
identifications could involve minimal polynomial relations which
are $q$-polynomials, i.e., of weak Frobenius type. Denoting the
matrix unit in the $(i,j)$ position of the block $B_{rs}$ as
$e^{i,j}_{r,s}$, and expanding these to a base of $B$, we have
some minimal quasi-linear relations of the form
$$\sum _{r,s}  \sum _{i,j=1}^{n_{r,s}} c_{ij}
(\lam^{i,j}_{r,s})^{q_{r,s}} = 0.$$ In this case we also say that
the gluing has {\bf weak Frobenius type}. Note that the $q_{r,s}$
are independent of $i$ and $j$, so we might as well assume that
$i=j=1$.

Let  $\Lambda\subset K[\lam_1, \dots, \lam_N]$ denote the set of
all weak $F$-Frobenius type polynomials  (see \Dref{FTdef}), where
$N = \dim_K(B)$. Let $\phi$ denote the map given by $a\mapsto
a^{q}$, where $q = \card{F}$ and where, as above, we take $q = 1$ if
$F$ is infinite. Noting that weak $F$-Frobenius type relations are
$F$-linear combinations of $\phi^j(\lam_i)$ for $j \geq 0$ and $i
= 1,\dots, m$, we may view $\Lambda$ as a module over $K[\phi]$,
under the obvious operation $\phi \cdot a = a^q$ ($a\in K$) and
$\phi\cdot \lam_i = \phi(\lam_i)$. In fact, $\Lambda$ is a free
module, spanned by $\lam_1,\dots,\lam_N$. When $F$ is finite,
$K[\phi]$ is isomorphic to the ring of polynomials in one variable
over $K$. For infinite $F$, $K[\phi] = K$, and $\Lambda$ is merely
the $K$-dual of $B$, as a vector space.

\begin{thm}
For any $F$-subspace $A \sub B$, the relations $\pol{A}$ form a
free module of rank at most $N$ over $K[\phi]$.
\end{thm}
\begin{proof}
Indeed, $\pol{A} \sub \Lambda$ by \Tref{linear}, and $\Lambda$ is
a free module over the principal ideal domain $K[\phi]$.
\end{proof}

We can improve this estimate, by noting that the weak Frobenius
relations depend on the Wedderburn block, not on the
indeterminate, so we can reduce $\Lambda$ to the module generated
by one representative indeterminate for each Wedderburn block
above the diagonal. This is $\binom {m}2$, where $m$ is the number
of diagonal Wedderburn blocks in the given representation of $A$
(clearly $m^2 \leq N$). Thus, we have

\begin{cor} The weak Frobenius relations of
a \Zcd\ $F$-subalgebra $A$ of $B$ can be defined by at most
$\binom {m}2$ polynomial relations, where each relation is
duplicated $\dim(B_{uv})$ times. (For example, one needs four
relations to define $\M[2](\F_q)$ inside the algebraic closure:
$\lam_{ij}^q = \lam_{ij}$ for $i,j = 1,2$.)
\end{cor}

We can improve this result even further.

\begin{lem}\Label{weakFrobe}
{\ } \begin{enumerate} \item In a polynomial relation of weak
Frobenius type, we may assume that one of the $q_{ij} = 1$.

\item Given two polynomial relations $f_1,f_2$ of weak Frobenius
type and given some $\lam_i$, we may modify these two polynomial
relations to assume that $\la _i$ appears in at most one of them.
\end{enumerate}
\end{lem}
\begin{proof}

(1)
 Otherwise we can write each $q_{ij} = q q_{ij}'$, and then
$$\left(\sum _{i=1}^m \sum _{j\ge 1} c_{ij} \lam_i^{q'_{ij}}\right)^q
=\sum _{i=1}^m \sum _{j\ge 1} c_{ij} \lam_i^{q_{ij}} = 0,$$ so
taking the $q$-root yields a polynomial relation of lower degree,
and we conclude by induction.

 (2) The argument is by induction on the degree $q^{d_i}$
of $\la _i$ in $f_1$ and $f_2$. Suppose $d_1 \ge d_2$. Then the
degree of $q$ in $f_1- f_2^{d_1-d_2}$ is less than $d_1$, so we
continue by induction.
\end{proof}
\begin{cor} Suppose the weak Frobenius relations of $A$ are
defined by $\mu \le \binom {m}2$ polynomial relations $\{f_1,
\dots, f_\mu\}$. Then we may take $\mu$ variables and assume that
each of them appears in at most one of $\{f_1, \dots, f_\mu\}$.
\end{cor}

\section{Generic \Zcd\ algebras and PI-generic
rank}\Label{sec:6}

Our study of \Zcd\ algebras has been motivated by the desire to
find an algebra in a given variety whose structure is especially
malleable. Since every variety contains a relatively free algebra,
we are led to study the relatively free algebras in the variety
generated by \Zcd\ algebras. These are the generic algebras, which
can be described in terms of ``generic'' elements. First we start
with the classical setting, and then we see how the presence of
finite fields complicates the situation and requires a more
intricate description.

We want to define a ``finitely generated'' generic algebra. This
means that we need to consider polynomial identities in only a
finite number of indeterminates. First, we need to clarify exactly
what we mean generically by ``generation.''

\begin{defn} The {\bf topological rank} of a \Zcd\ algebra $A$ is defined as
the minimal possible number of generators of an $F$-subalgebra
$A_0$ of $A$ for which $\cl{A_0} = A$.\end{defn}

\begin{rem}\Label{infgen11}
By Theorem \ref{semis}, every semiprime \Zcd\ algebra is a finite
direct sum of simple algebras, and thus has finite topological
rank.
\end{rem}

 Thus, the obstruction to finite topological rank
is found in the radical.

\begin{exmpl}\Label{infgen} Let $K$ be an infinite dimensional
field extension of  a finite field $F$, and consider $$ A = \left(\AR{F & K \\
0 & F}\right).$$ This \Zcd\ algebra has infinite topological rank,
since any finite number of elements generates only a finite
subspace of $K$ in the $1,2$ position, which is \Zcd.
\end{exmpl}

Nevertheless, we do have the following information.

\begin{rem}
When the Peirce idempotent $\tilde e$ has infinite typ,
(cf.~\Dref{idemptype}), then the spaces $A\tilde e$ and $\tilde e A$
have finite topological rank, since they are naturally vector
spaces over $K$ (although this is not the structure of the initial
vector space). The action is via the isomorphism of $K$ with the
center of the prime component corresponding to $\tilde{e}$.
\end{rem}

\subsection{Generic algebras over an infinite
field}

Clearly, the topological rank of a f.d.~algebra $A$ over field $K$
is not greater than its dimension, which is given to be finite.
Thus, in this case, we need only finitely many  elements to define
the generic algebra.

\begin{constr}\Label{classicgen}
The classical construction of a generic algebra of a f.d.~algebra
$A$ over an infinite field $F$ is to take a base $b_1, \dots, b_n$
of $A$ over $F$, adjoin indeterminates $\xi _{i}^{(k)}$ to~$F$ ($i
= 1,\dots,n$, $k \in \N$), and let $A'$ be the algebra generated
by the ``generic'' elements $Y_k = \sum _{i=1}^n \xi_{i}^{(k)}
b_i$, $k \in \N$. It is easy to see \cite[Example 3.26]{BR} that
in the case where $F$ is infinite, $A'$ is PI-equivalent to $A$, and in fact
$A'$ is relatively free in the variety defined by $\id (A)$. The
most celebrated example in PI-theory is when $A = \M[n](F)$, the
algebra of $n \times n$ matrices. Then $A'$ is the algebra of
generic $n \times n$ matrices, generated by the {\bf generic
matrices} $Y_k = (\xi _{ij}^{(k)})_{ij}$, $k \in \N$.
\end{constr}

\begin{exmpl}\Label{generic1} The generic upper triangular matrix
$Y_k = \left(\AR{\xi_{1}^{(k)} & \xi_{2}^{(k)} \\ 0 &
\xi_{3}^{(k)}}\right)$ is defined over the polynomial algebra $C=
F[\xi_{j}^{(k)} : j = 1,2,3 \, ,\,k \in \N].$ We get the {\bf
generic algebra of upper triangular matrices} by taking the
subalgebra of $\M[n](C)$ generated by the $Y_k$.
\end{exmpl}

\begin{constr}\Label{generic1.1}
Alternatively, when building generic elements for an arbitrary
f.d.~algebra $A$ over an infinite field, we could take the powers
of the radical $J$ into account. Writing  $A = S \oplus J$ by
Theorem~\ref{Zarcl1}, where $J^\nu = 0$, we can view $J^2$ as
those blocks at least two steps above the diagonal, i.e., lying in
$\sum _{r+2 \leq s} B_{rs}$. We take a generic algebra for $S$ and generic elements for $J/J^2$ (which can be viewed as a
complementary subspace for $J^2$ inside $J$), taking identical
gluing into account; these then yield generic elements for $A$.
This also will be the approach that we take for \Zcd\ algebras
over arbitrary fields.
\end{constr}

\subsection{PI-generic rank over an arbitrary field}

Since the topological rank could be infinite, we look for an
alternative concept which is more closely relevant to PI-theory.

\begin{defn} The {\bf PI-generic rank} of $A$ is the minimal number
$m$ of elements needed to generate a subalgebra satisfying the
same PI's as $A$. Then the relatively free PI-algebra of $A$ could
also be generated by $m$ elements. In the literature, the
PI-generic rank is often called the {\bf basic rank}.\end{defn}

Clearly, the PI-generic rank is less than or equal to the
topological rank.

\begin{exmpl}
$$ A = \left\{\left(\AR{\a & \beta & \gamma \\0 & \a & \beta \\
0 & 0 & \a}\right): \a\in F, \beta \in K\right\}$$ is a
commutative algebra, having PI-generic rank 1, but having infinite
topological rank when $K$ is infinite dimensional over $F$. This
example also shows that gluing can lower the PI-rank.
\end{exmpl}

 When computing PI-generic rank, we study a
polynomial in terms of substitutions of its monomials, as usual,
which can be complicated by the absence of indeterminates in
certain monomials. Accordingly, recall from
\cite[Definition~2.3.15]{Row1} that a polynomial $f$ is called
{\bf blended} if each monomial appearing in $f$ must appear in
each monomial of $f$. As noted in \cite[Exercise~2.3.7]{Row1}, any
PI is a sum of blended PI's, seen by specializing each
indeterminate to 0 in turn. Thus, we can limit ourselves to blended
PI's when determining the PI-generic rank.

Let us consider the PI-generic rank of a \Zcd\ PI-algebra.
Although when $F$ is infinite, this is obviously finite (since the
variety contains a finite dimensional algebra), the situation
becomes much more interesting when $F$ is finite. Although, as
already seen in Example~\ref{infgen}, we might need infinitely
many generic elements to generate our algebra, we aim to show,
however, that the PI-generic rank is always finite.

\begin{thm}\Label{frank}
Any \Zcd\ algebra $A$ (over an arbitrary field) has finite
PI-generic rank.
\end{thm}
\begin{proof}
(As noted above this statement is trivial for algebras over
infinite fields.) We decompose $A = S \oplus J$, where $S$ is
semisimple and $J$ is nilpotent, of nilpotence index $\nu$. By
\Pref{break}, it is enough to consider specializations for which
every variable takes values either in $S$ or in $J$.

We only need to consider blended polynomials. In any nonzero
evaluation of a polynomial on~$A$, at most $\nu-1$ components can
belong to $J_1$, as defined in Remark~\ref{gluetype1}. But $J_1$,
being a variety, has a finite number $\psi$ of irreducible
components, each of which has a generic element which we can use
for the radical substitution. Our ``generic'' radical substitution
could be taken to be the sum of these substitutions.

This leaves the semisimple substitutions, which we consider
``layered'' around the radical substitutions. The PI-generic rank
of $S$ is less than or equal to its topological rank, which is
finite; cf.~\Rref{infgen11}.

In view of $A = S \oplus J$, we see that  the PI-generic rank is
at most $\mu + \nu-1 $, where $\mu$ is the PI-rank of $S = A/J$.
\end{proof}

By passing to the \Zcr, we have:

\begin{cor}
Any representable algebra $A$ (over an arbitrary field) has a
PI-equivalent algebra with finite PI-generic rank.\end{cor}

\begin{rem} We can improve the bound given in Theorem \ref{frank}.
First, any central simple algebra over an infinite field is
generated by two elements; thus, its topological rank (and thus
PI-generic rank) is $2$.  Thus, when $F$ is infinite, $\mu = 2,$
so our bound becomes $\nu+1$. (This can be lowered even further,
since the \Zcr\ of a one-generator algebra contains its radical
part.)

Over a finite field $F$, any simple algebra has the form
$\M[n](F)$. If $|F|>n$, then $\M[n](F)$ is generated by two
elements, one being the diagonal with distinct entries and the
other being the upper triangular matrix $\sum _{i=1}^{n-1}
e_{i,i+1}$. On the other hand, when $|F|^2<n,$ the topological
rank starts growing, because of repeating eigenvalues, although
obviously the topological rank is finite (bounded by $n^2$) and
in fact grows much more slowly, bounded by  $2+ \log _{|F|} n,$ as
seen by the argument given above. At any rate, the topological
rank, and thus the PI-rank, of any central simple algebra of
dimension $n^2$ is finite, bounded by some function of $n$.
Arguing by components shows that the PI-generic rank of any
noncommutative semisimple algebra is given along the same lines.

When $F$ is a finite field of $q$ elements, there are only
finitely many possible elements in $\M[n](F)$, namely $q^{n^2}$,
and thus $q^{2n^2}$ possible ordered pairs. If we have more
components in $A$, some pair must repeat itself, and thus the
corresponding components become glued when $\psi \ge q^{n^2}$
unless we have a greater number of generators. Thus the size of
$\psi$ can also force up (logarithmically) the bound for the
PI-generic rank.
\end{rem}

\subsection{Generic representable algebras, not over an infinite
field}

Our main theorem in this section is that for any representable
algebra $A$ there exists a relatively free, finitely generated
algebra in the variety $\operatorname{Var}(A)$ obtained by the
identities of $A$.

 Although, for any representable algebra $A$ over an infinite
field $F$, the classical construction of a generic algebra for $A$
is PI-equivalent to the original algebra ~$A$, this is no longer
the case when $F$ is finite (or even worse, when there is no base
field). Thus, when considering finite characteristic, we need to
introduce new commutative rings that need not be fields.

\begin{exmpl}\Label{generic1.5} Suppose the field $F$ contains $q$
elements. Then $F$ is not PI-equivalent to the ring of polynomials
$F[\xi]$, so we must pass to $F[\xi]/\langle \xi ^q -\xi\rangle$,
where the image $\overline {\xi}$ of $\xi$ is a generic element.
Note that $F[\xi]/\langle \xi ^q -\xi\rangle$ is isomorphic to a
direct product of $q$ copies of $F$, so another way of viewing our
generic element is as a $q$-tuple listing the elements of $F$.
Unfortunately, this may not suffice to describe the PI's since
they involve more than one substitution. For two generic elements
we need to pass to $$F[\xi_1, \xi_2]/\langle \xi_1 ^q -\xi_1 ,
\xi_2 ^q -\xi_2\rangle,$$ which is isomorphic to a direct product
of $q^2$ copies of $F$, and so on. Of course, since the identity
of commutativity only requires two variables, this is enough for
the generic element of the variety of $F$, but we already see the
difficulty arising of predicting how many generic elements we need
to construct the generic algebra.
\end{exmpl}

Nevertheless, there is a way to define the generic algebra for a
general \Zcd\ algebra $A = \cl{A}$ (represented in some
f.d.~algebra $B$ over an algebraically closed field $K$).  The
idea is to define everything as generically as possible.

\begin{constr}[General construction of generic  algebras]\Label{genalgfin}

Letting $\CS_1, \dots, \CS_t$ denote the irreducible components of
$\cl{A}$ under the Zariski topology, suppose each $\CS_i$ is
defined over a field with $q_i$ elements. Then we need $s$
``mutually generic'' elements $b_{i1}, \dots, b_{is}$ in each
component. Towards this end, we take a generic element
$$b \in \CS_1^{s} \times \dots \times \CS_t^{s},$$
where each $\CS_i^{s}$ denotes the direct product of $s$ copies of
$\CS_i$.  Thus $b$ has the form $((b_{11}, \dots, b_{1s}),
(b_{21}, \dots, b_{2s}), \dots, (b_{\mu 1}, \dots, b_{\mu s}))$,
where each $(b_{i1}, \dots, b_{is})\in \CS_i;$ by definition, the
$b_{ik}$ are ``mutually generic''.
 Next, we define the {\bf generic
coefficient ring}
$$C = F[ \xi _{ik}\! : 1 \le i \le s, \ 1 \le k \le \mu ]/\langle
\xi_{ik}^{q^{d_i}}-\xi_{ik}:  1 \le i \le s, \ 1 \le k \le \mu
\rangle,$$ and the generic elements $Y_k = \sum _{i=1}^s \bar
\xi_{ik} b_{ik}$ ($k = 1,\dots,\mu$), where $\bar \xi_{ik}$ is the
image of $\xi_{ik}$ in $C$. The subalgebra of $B$ generated by
the~$Y_k$ serves as our generic algebra for the variety generated
by $A$.

Note that this construction is completely general since the free
algebra of any variety is representable; however, this is difficult
to prove (using the theory developed in this paper).
\end{constr}

\begin{thm}\label{mainthm1} The algebra $\mathcal Y = F\{Y_1, \dots, Y_\mu\}$ of
Construction~\ref{genalgfin} is relatively free in $\VarF A$, with
free generators $Y_1, \dots, Y_\mu$. If $t \ge
\operatorname{PI-rank} A$, then $\VarF {\mathcal Y} = \VarF A$.
\end{thm}
\begin{proof}
Any set of mutually generic elements specialize to arbitrary
elements of $A$, so it remains to show that $\mathcal Y$ satisfies
the identities of $A$. But this is clear, since any element is the
sum of generic components, which satisfy the identities of $A$ by
definition.
\end{proof}

This construction also works for nonassociative \Zcd\ algebras of
arbitrary signature, in the framework of universal algebra.

\begin{rem}
One cannot simply describe the generic algebra in Construction~
\ref{genalgfin} by taking the $b_i$ from a base of the extended
algebra as in Construction~\ref{classicgen}, because the
Gel'fand-Kirillov (GK) dimensions do not match, as evidenced by
Belov's computation discussed in Remark~\ref{GZB1}. Indeed, the
generic coefficient ring $C$ is finite, and thus the GK dimension
would only be equal to the GK dimension of the first block,
instead of the sum of the GK dimensions of the blocks.
\end{rem}

\begin{exmpl}
Let $A$ be the algebra of triangular matrices with entries as
follows: $\left(\AR{\a & \beta
\\ 0 & \gamma   }\right)$ where $\a $ is
in the finite field $F$, and $\beta ,\gamma $ are in an infinite
field extension $K$ of $F$. Continuing the notation of
Construction~\ref{genalgfin}, the generic algebra is generated by
matrices of the form $Y_i = \left(\AR{\xi_{i1} & \xi_{i2}
\\ 0 & \xi_{i3}   }\right)$
where $\xi_{i1}$ is a generic element of $C$, whereas
$\xi_{i2},\xi_{i3}$ are indeterminates over $K$.
\end{exmpl}

\begin{exmpl} The generic upper triangular matrices for an algebra with
Frobenius gluing of power $q$ along the diagonal  can be written
in the form $Y_k = \left(\AR{\xi_{1k} &  \xi_{2k}
\\ 0 & \xi_{1k}^q   }\right)$.\end{exmpl}

When Frobenius gluing is involved, we can still use the generic
coordinate ring of Construction~\ref{genalgfin}. The generic
description of partial gluing up to infinitesimals becomes more
complicated in the presence of Frobenius gluing, because we need
to deal both with the Frobenius automorphism and also with the
degree of the infinitesimal.

We close by providing an explicit construction for generic
PI-algebras of \Zcd\ algebras. As in
Construction~\ref{generic1.1}, we take the powers of the radical
$J$ into account.

\begin{constr}[The explicit generic algebra of a
\Zcd\ algebra of finite topological rank]\Label{generic2}

In view of Construction~\ref{classicgen}, we may assume the base
field $F$ is finite. The construction requires generic elements
for the \Zcd\ algebra $A$, which will be defined over the ring of
polynomials $F[\xi_1,\xi_2,\dots]$ (with an appropriate indexing).
There are several methods; we choose the one that is perhaps most
intuitive according to the structure, but rather intricate. Our
point of departure is the Wedderburn Block form
(Theorem~\ref{block1}). Namely, write $A \subset \M[n](K)$, where
$F \sub K$.
\begin{enumerate}
\item First consider the center $A_0$ of a simple component in $A$, such as
$$A_0 = \set{\smat{\alpha}{0}{0}{\alpha^p} \co \alpha \in K}.$$ This
algebra may be described by the Frobenius gluing of $1\times 1$
blocks along a single gluing component. Namely, the algebra has
the form $\set{\sum_{i=1}^{s} \alpha^{\phi_i} e_i \suchthat \alpha
\in K}$, where $e_i$ are the basic idempotents and $\phi$ is the
exponent vector, taking $q$-power values (for $q = \card{F}$).
The generic elements can be taken as $X_k = \sum_{i=1}^{s}
\xi_k^{\phi_i}
e_i$. %
\item Let $S$ be a simple component of $A$. In $B \sub \M[n](K)$,
$S$ is contained in a direct sum of matrix blocks of the same
size. Let $e_i^{jj'}$ denote the $1\times 1$ matrix units in the
\th{i} block, whose corresponding block idempotent is $e_i$.
Keeping the notation as above, the generic elements can be taken
as $X^{k,jj'} = \sum_{i=1}^{s} (\xi_k^{jj'})^{\phi_i} e_i^{jj'}$,
where the variables are $\xi_k^{jj'}$. For example, we could have
$$\hat{X}_{[1]}^{k,21} =
\smat{\smat{0}{0}{\xi_k^{2,1}}{0}}{0}{0}{\smat{0}{0}{(\xi_k^{2,1})^q}{0}}.$$ %
\item Let $d$ denote the number of glued components in $A$, and let
$\hat{e}_u$ denote the idempotent corresponding to the \th{u}
component, decomposed as $\hat{e}_u = \sum_{r \in T_u} e_r$, as in
Subsection~\ref{ss:sP}. Let $\hat{X}_{[u]}^{k,jj'}$ denote the
glued sum of appropriate powers of $\xi_k^{jj'}$, placed in the
$(j,j')$ entry of the glued blocks, where the sum is over the
blocks $r \in T_u$. Each generic element of this type can be
decomposed as a sum $\hat{X}_{[u]}^{k,jj'} = \sum_{r \in T_u}
X_r^{k,jj'}$, where $X_r^{k,jj'}$ is the appropriate power of
$\xi_k^{jj'}$, placed in the $(j,j')$ entry of the \th{r} block.
The $X_{[u]}^{k,jj'}$ are the semi-simple part of our generic
elements.

Let $b_1,\dots,b_\tau$ be a topological basis for $S$. We continue
as in \Rref{abitmore}. For every $2 \leq t \leq \nu$ (where $\nu$
is the nilpotence index of $J$; see \Pref{boundonlen}), for all
indices $1 \leq u_1,\dots,u_t \leq d$, and for every (fully
refined) equivalence class $\gamma \sub T_{u_1,\dots,u_t}$ (see
\Dref{eqdef} and \Rref{gluetype1}), we take all the elements
$$X_{\vec{k}}^{*,\gamma} = \sum_{(r_1,\dots,r_t) \in \gamma} {X_{r_1}^{k_1,j_1j_1'} b_{p_1} X_{r_2}^{k_2,j_2j_2'}
\cdots X_{r_{t-1}}^{k_{t-1},j_{t-1}j_{t-1}'} b_{p_{t-1}}
X_{r_t}^{k_t,j_tj_t'}},$$ where each of $p_1,\dots,p_t$ ranges
over the values $1$ through $\tau$, each $k_i$ ranges over~$\N$;
and each $(j_i,j_i')$ ranges over the matrix entries of the blocks
in the $r_i$ place (they have the same dimensions for each
$(r_1,\dots,r_t) \in \gamma$).
\end{enumerate}
\end{constr}

\begin{thm}\label{mainthm2}
The algebra $\mathcal Y = F\set{X_{\vec{k}}^{*,\gamma} \suchthat
k_1,\dots,k_t \leq \mu}$ of Construction~\ref{generic2} (where we
range over all possible choices of $t$, $p_1,\dots,p_t$ and
$j_1,j_1',\dots,j_t,j_t'$ and $\gamma$) is relatively free in
$\VarF A$, with free generators $\tilde X_{\vec{k}}^{*,\gamma}$.
If $\mu \ge \operatorname{PI-rank} A$ (in particular, if $\mu$ is
at least the number of topological generators of $A$), then $\VarF
{\mathcal Y} = \VarF A$.
\end{thm}
\begin{proof} Same as \Tref{mainthm1}.
\end{proof}

\begin{cor} For any representable algebra $A$, its variety has a
finitely generated relatively free algebra. \end{cor}
\begin{proof}
In view of Theorem \ref{frank}, we may assume that $A$ has finite
PI-rank, and thus we are done by Theorem
\ref{mainthm2}.\end{proof}

This construction is needed in a forthcoming paper, where we
describe varieties of algebras in terms of a certain kind of
quiver.


\begin{thebibliography}{99}
\bibitem[B1]{Bel1} %
Belov, A.Ya., {\it Algebras with polynomial identities:
Representations and combinatorical methods}, Doctor of Science
Dissertation, Moscow (2002).

\bibitem[B2]{Bel2} Belov, A., {\it Counterexamples to the Specht problem},
Sb. Math. {\bf 191} (2000), pp. 329-340.

\bibitem[BR]{BR} %
Belov, A.Ya. and Rowen, L.H. ``Computational aspects of polynomial
identities''. Research Notes in Mathematics, 9. AK~Peters, Ltd.,
Wellesley, MA, 2005.

\bibitem[GZ]{GZ}
Giambruno, A. and Zaicev, M., \textit{Minimal varieties of
algebras of exponential growth},  Adv. Math.  {\bf174}  (2003),
pp. 310-323.

\bibitem[H]{Hum}  Humphreys J., \textit{Linear algebraic
groups}, Springer Lecture Notes in Mathematics \textbf{21},
(1975), Springer, New York.

\bibitem[KMT]{KombMiyanMasayoshi}
Kambayashi, T.; Miyanishi, M. and Takeuchi, M., ``Unipotent
algebraic groups'', Lecture Notes in Mathematics, {\bf 414},
Springer-Verlag, Berlin-New York., 1974.

\bibitem[L]{Lew74} Lewin, J., \textit{A matrix representation
for associative algebras.\!~I and II}, Trans. Amer. Math. Soc.
\textbf{188}(2), 293--317, (1974).

\bibitem[M]{Miyanishi}
Miyanishi, M., 
{\it Questions of rationality of solvable algebraic groups over
non-perfect fields}, {Annali Mat. Pura Appl.}, {\bf 61}(4),
97--120, (1963).

\bibitem[P]{Put}  Putcha M., \textit{Linear algebraic
monoids}, London  Math. Soc. Lecture Notes \textbf{133},
Cambridge University Press, 1988.

\bibitem[R1]{Row1} %
Rowen, L.H., \textit{Polynomial Identities in Ring Theory}, Academic
Press Pure and Applied Math, {\bf 84}, New York, 1980.

\bibitem[R2]{Row2} %
\bysame, \textit{Ring Theory I}, Academic Press Pure and Applied
Math, {\bf 127}, New York, 1988.

\bibitem[R3]{Row3} %
\bysame, \textit{Algebra: A noncommutative view}, Amer. Math. Soc.
Graduate Series in Math., 2008.

\bibitem[T]{Tits}
Tits J., ``Lectures on algebraic groups'', Dept. of Math., Yale
Univ. -- New Haven, 1966/67.

\end{thebibliography}
\end{document}